\documentclass[12pt]{amsart}
\usepackage{amsmath,amssymb,amsthm}

\usepackage[dvips]{graphicx}
\usepackage{color}
\usepackage{wasysym}

\theoremstyle{plain}
\newtheorem{prop}{Proposition}[section]
\newtheorem{lem}[prop]{Lemma}

\newtheorem{thm}[prop]{Theorem}

\newtheorem{remark}[prop]{Remark}
\newtheorem{definition}[prop]{Definition}

\numberwithin{equation}{section}

\newcommand{\N}{{\mathbb N}}

\newcommand{\R}{{\mathbb R}}

\definecolor{blu}{rgb}{0,0,1}

\newcommand{\mE}{\mathcal{E}}
\newcommand{\mP}{\mathcal{P}}
\newcommand{\p}{\psi}
\newcommand{\e}{\varepsilon}
\newcommand{\lan}{\langle}
\newcommand{\ran}{\rangle}

\newcommand{\vertiii}[1]{{\left\vert\kern-0.25ex\left\vert\kern-0.25ex\left\vert #1
    \right\vert\kern-0.25ex\right\vert\kern-0.25ex\right\vert}}

\begin{document}
\title[traveling waves for Gross-Pitaevskii equation]{Finite energy traveling waves for the Gross-Pitaevskii equation in the subsonic regime}

\author{Jacopo Bellazzini}
\address{J. Bellazzini,
\newline  Universit\`a di Sassari, Via Piandanna 4, 70100 Sassari, Italy}%
\email{jbellazzini@uniss.it}%

\author{David Ruiz}
  \address{David Ruiz \\
    Universidad de Granada\\
    Departamento de An\'alisis Matem\'atico\\
    Campus Fuentenueva\\
    18071 Granada, Spain}
  \email{daruiz@ugr.es}

\thanks{J. B. is partially supported by Project 2016 ÒDinamica di equazioni nonlineari dispersiveÓ of FONDAZIONE DI SARDEGNA. D. R. has been supported by the FEDER-MINECO Grants MTM2015-68210-P and PGC2018-096422-B-I00, and by J. Andalucia (FQM-116).}

\begin{abstract} In this paper we study the existence of finite energy traveling waves for the Gross-Pitaevskii equation. This problem has deserved a lot of attention in the literature, but the existence of solutions in the whole subsonic range was a standing open problem till the work of Mari\c s in 2013. However, such result is valid only in dimension 3 and higher. In this paper we first prove the existence of  finite energy traveling waves for almost every value of the speed in the subsonic range. Our argument works identically well in dimensions 2 and 3. 
	
With this result in hand, a compactness argument could fill the range of admissible speeds. We are able to do so in dimension 3, recovering the aforementioned result by Mari\c s. The planar case turns out to be more difficult and the compactness argument works only under an additional assumption on the vortex set of the approximating solutions. 
	
\end{abstract}

\maketitle

\section{Introduction}

In this paper we are concerned with the Gross-Pitaevskii equation

\begin{equation}\label{eq:GP}
i \partial_t \Psi=\Delta \Psi+\Psi\left(1-|\Psi|^2\right)  \text{on } \R^d \times \R
\end{equation}
when $d=2$ or $d=3$. Observe that this is no more than a Nonlinear Schr\"{o}dinger Equation with a Ginzburg-Landau potential. The Gross-Pitaevskii equation was proposed in 1961 (\cite{gross, pita}) to model a quantum system of bosons in a Bose-Einstein condensate, via a Hartree-Fock approximation (see also \cite{b1, b2, jpr1, jpr2}). It appears also in other contexts such as the study of dark solitons in nonlinear optics (\cite{k1, k2}).

From the point of view of the dynamics, the Cauchy problem for the Gross-Pitaevskii equation  was first studied  in one space dimension by Zhidkov \cite{Z} and in dimension $d=2,3$
by B\'{e}thuel and Saut \cite{bs} (see also \cite{ge1,ge2, killip}). At least formally, equation \eqref{eq:GP} presents two invariants, namely:
\begin{itemize}
	\item \emph{Energy:} 
	\[
	\mathcal{E}= \int_{\R^d} \frac 12 |\nabla \Psi|^2 +\frac 14  \left(1-|\Psi|^2\right)^2,
	\]
	\item \emph{Momentum:} 
	\[
	\mathcal{\bf{P}}=\frac 12 \int_{\R^d} \lan i \nabla \Psi, \Psi \ran,
	\]
	where $ \lan f,g \ran=Re(f)Re(g)+Im(f)Im(g)$.
\end{itemize}

This paper is focused on the existence of traveling wave solutions to \eqref{eq:GP}, that is, solutions in the form
\begin{equation}\label{eq:ansatz}
\Psi(x,t)=\psi(x_1-ct, \tilde{x}), \ \ \tilde{x}=(x_2 \dots x_d) \in \R^{d-1},
\end{equation}
where the parameter $c\in \R$ characterizes the speed of the traveling wave. 
Without any lack of generality we will consider $c>0$ throughout the paper. By the ansatz \eqref{eq:ansatz} the equation for the profile $\psi$ is given by
\begin{equation}\label{eq:GPellip}
i c \,\partial_{x_1}\psi +\Delta \psi+\left(1-|\psi|^2\right)\psi=0 \ \ \mbox{in } \R^d.
\end{equation}

%We notice that $\mathcal{P}$ can also be written as
%$$\mathcal{P}:=\int_{\R^2} \left(-u(x)\nabla v(x)+v(x)\nabla u(x) \right) dx$$
%where $u:=Re(\Psi), v:=Im(\Psi)$.

The study of finite energy traveling waves for \eqref{eq:GP} has also implications in the dinamics of the equation. In particular, their pressence is an obstruction to scattering of solutions. Scattering of small energy solutions has been proved in \cite{gnt,gnt2} for $d=3$, and such result is not true in dimension $d=2$. This latter fact may seem surprising for a defocusing Schr\"{o}dinger equation; the reason is that finite energy solutions of \eqref{eq:GP} do not vanish at infinity. 

Nontrivial finite energy traveling waves in dimension $d=1$ are explicitly known, and they are uniquely given (up to rotation or translation) by the expression
$$\psi_c(x)=\sqrt{\frac{2-c^2}{2}} \tanh \left( \frac{\sqrt{2-c^2} }{2}x \right)+i\frac{c}{\sqrt{2}},$$
if $c<\sqrt{2}$. In the literature the function $\psi_0$ is called black soliton whereas $\psi_c$ ($c \neq 0$) receives the name of dark soliton. Their orbital and asymptotic stability has been studied, see \cite{bgss, bgs}.

\medskip 
The problem of finding solutions to \eqref{eq:GPellip} in dimension $d\geq 2$ has a long story. In the pioneer work of Jones, Putterman and Roberts (\cite{jpr1, jpr2}), formal calculations and numerical analysis gave rise to a set of conjectures regarding existence, asymptotic behavior and stability of finite energy travelling waves: the so-called the Jones-Putterman-Roberts program. In particular,  the existence  of finite energy traveling waves is expected if and only if $c \in (0, \sqrt{2})$ (the sub-sonic case). The threshold value $c= \sqrt{2}$ comes from the linearization of the problem around the constant solutions of modulus 1. In a certain sense, those solutions correspond to local minima if $c< \sqrt{2}$.

In the last years much progress has been made to give rigorous proofs of those conjectures. Nontrivial finite energy traveling waves for supersonic speed $c>\sqrt{2}$ do not exist, see \cite{gravejat-CMP}. In dimension $d=2$ this nonexistence result holds also for $c=\sqrt{2}$, see \cite{gravejat-DIA}. For general nonlinearities analogous results have been proved in \cite{Maris-SIAM}.

Concerning the asymptotics of finite energy solutions, for any $d \geq 2$, finite energy solutions of \eqref{eq:GPellip} converge at infinity to a fixed complex number of modulus $1$. By the phase invariance of the problem, we can assume that
\begin{equation} \label{limit1} \psi(x) \to 1 \mbox{ as } |x| \to +\infty. \end{equation}
A more precise asymptotic description of $\psi$ is indeed available, see \cite{gravejat-AIHP, gravejat-Asymp, gravejat-Adv.Diff}.

A very active field of research is the study of the location and dynamics of vortices, namely, the zeroes of the wave function $\psi$. The existence of multi-vortices traveling waves with small speed has been proved in dimension $d=2$, see \cite{ lw, chr2, chr3}. In dimension $3$ there are traveling vortex rings (\cite{lwy}) as well as leapfrogging vortex rings, see \cite{jer2}.

\medskip 

At least formally, the Lagrangian associated to \eqref{eq:GPellip} is defined as:

\begin{equation}\label{eq:ac}
I^c(\psi)=  \mathcal{E}(\psi) - c \mathcal{P}(\psi)= \frac{1}{2} \int_{\R^d} |\nabla \psi|^2  -c\mathcal{P}(\psi)+ \frac 14  \int_{\R^d} \left(1-|\psi|^2\right)^2,
\end{equation}
where $\mathcal{P}$ is the first component of the momentum $\mathcal{\bf P}$ that, under suitable integrability conditions (and taken into account \eqref{limit1}) can be written as:
\begin{equation}\label{eq:MO}
\mathcal{P}(\psi):= -\int_{\R^d}  \partial_{x_1} (Im \Psi)  (Re \Psi-1).
\end{equation}

A classical approach to prove existence of traveling waves (starting from \cite{jpr1, jpr2}) is a minimization procedure
of the energy functional $\mathcal{E}$ under the constraint $\mathcal{P}(\psi)=p$ in a suitable functional space. This approach has been pursued in a number of papers, see for instance \cite{bgs-cmp, bos} for the Gross-Pitaevskii equation and \cite{chr-mar} for more general nonlinearities. A major difficulty in this strategy is to find a natural definition of the momentum for functions with finite energy, since the integrand in \ref{eq:MO} might be non-integrable (see \cite{bos}). This approach has the advantage of providing orbital stability of the solutions found (more precisely, of the set of minimizers). As a drawback, the speed $c$ appears as a Lagrange multiplier and is not under control. In particular the possibility of gaps in the subsonic range of velocity cannot be excluded with the constrained minimization approach (see \cite{bgs-cmp}).

We shall also quote existence results for small values of $c$, see \cite{bs} in dimension 2 and \cite{chr} in dimension 3, but a complete existence result in the sub-sonic case remained for many years as a standing open problem. Finally, Mari\c s proved in \cite{Ma} the existence result for any $c \in (0, \sqrt{2})$ in dimension $d \geq 3$. His approach is, summing up, to minimize $I^c(\psi)$ under a Pohozaev-type constraint. Once this is accomplished, Mari\c s proves that the corresponding Lagrange multiplier is 0, concluding the proof. This approach works also for more general nonlinearities with nonvanishing conditions at infinity, such as the cubic-quintic nonlinearity. As commented in \cite{Ma}, this minimization approach breaks down in dimension 2 because of different scaling properties: the infimum is $0$ and is never attained.

One important tool in Mari\c s' argument is the use of the fiber $t \mapsto u_t$, where $u_t(x_1, \tilde{x}) = u(x_1, t \tilde{x})$. For instance, in dimension $d \geq 4$ all solutions correspond to a maximizer of $I^c$ with respect to that fiber. In dimension $3$, $I^c(u_t)$ is independent of $t$ for any solution: the argument needs to be adapted, but still the use of the fiber is essential. Those cases have an analogy in the study of the Nonlinear Sch\"{o}dinger equation, see \cite{blions}, \cite{bgk}, respectively. However, in dimension $2$ this approach breaks down, and the fiber $u_t$ seems of no use; $I^c(u_t)$ attains a maximum at $t=1$ for any solution $u$.

One of the main motivations of this paper is to deal with the physically relevant $2D$ model where the existence of finite energy traveling waves in the full subsonic range is still an open problem. Our main result is the following:

\begin{thm} \label{teo:almost} There exists a subset $E \subset (0,\sqrt{2})$ of plein measure such that, for any $c \in E$, there exists a nontrivial finite energy solution of \eqref{eq:GPellip} $\psi_c$ such that:
	
	\begin{enumerate}
		\item For any $c_0 \in (0, \sqrt{2})$ there exists $\chi=\chi(c_0)>0$ such that $$0 < 
				I^c(\psi_c) \leq \chi \ \mbox{ for all } c \in E, \ c \geq c_0;$$ 
		\item $ind (\psi_c) \leq 1$, where $ind (\psi_c)$ stands for the Morse index of $\psi_c$, that is,
		
		$$ \sup \{dim\, Y: \ Y \subset C_0^{\infty}(\R^d) \mbox{ vector space, } (I^c)''(\psi_c)(\phi,\phi) <0  \ \forall \, \phi \in Y \} \leq 1.$$
	\end{enumerate}

\end{thm}

The proof deals directly with the Lagrangian $I^c$ and is focused on searching critical points by using min-max arguments. Our proofs use several ingredients: 

\begin{itemize}
	\item Several regularization (or relaxation) techniques have been used in the literature to deal with the Gross-Pitaevskii equation (\cite{bs, Ma}). Alternatively, some authors have proposed an approach by approximating domains, like flat tori, see \cite{ bgs-cmp, bos}. In this paper we choose the second approach, but we use as approximating domains the slabs:
	\begin{equation} \label{omegaN} \Omega_N=\left\{(x_1,\tilde{x})\in \R\times \R^{d-1}, \ \ \ -N<x_1<N\right\}, \ N \in \N. \end{equation}
	
	In other words, we first use a mountain-pass argument to address the question of existence of solutions to the problem:
	\begin{equation}\label{eq:GPstrip}
	\begin{array}{rcr} i c\partial_{x_1}\psi +\Delta \psi+\left(1-|\psi|^2\right)\psi & = & 0  \ \  \text{ on } \Omega_{N}, \\ \psi & = & 1 \text{ on } \partial \Omega_{N}. \end{array} \end{equation}
	The boundary condition is motivated by \eqref{limit1}. This approach has several advantages. First, as $\Omega_N$ is bounded in the $x_1$ direction, Poincar\'{e} inequality holds and we can work on the space $1+H_0^1(\Omega_N)$. As a consequence the momentum given by formula \eqref{eq:MO} is well defined. Secondly, as $\Omega_N$ is invariant along the variable $\tilde{x}$, a Pohozaev type inequality is satisfied without boundary terms (see Lemma \ref{lem:poho2}). This allows us to avoid the problem of unfolding choices of tori, as in \cite{bgs-cmp, bos}. %Third, the spaces $H_0^1(\Omega_N) \subset H_0^1(\Omega_{N'}) $ if $N<N'$, and this implies monotonicity of the related min-max values with respect to $N$.

	\item A second fundamental tool is an energy bound argument via monotonicity in order to control the energy of (PS) sequences for almost all values of $c$. This idea has been used many times in literature starting from \cite{struwe}. The main point here is that we are able to obtain a \emph{uniform bound on the energy for a subsequence of enlarging slabs $\Omega_{k(N)}$}. This is based in a key analytic argument, and it is fundamental in what follows. To the best of our knowledge, this abstract argument is completely new and could be of use in other frameworks where a monotonicity argument is used together with a relaxation procedure.
	
	\item The next step is to pass to the limit, and for that we need to deal with the problem of vanishing. Here we rely on arguments of \cite{bgs-cmp}, and we use in an essential way that $\psi_N$ are solutions of \eqref{eq:GPstrip}. 	
	We can also exclude the concentration of solutions near the boundary of $\Omega_N$, since the problem posed in the half-space
	$$
	\begin{array}{rcl}i c\partial_{x_1}\psi +\Delta \psi+\left(1-|\psi|^2\right)\psi & = & 0  \ \  \text{ on } \R^d_+, \\
	\psi & = & 1 \  \text{ on } \partial \R^d_+, \end{array}
	$$
	does not admit nontrivial solutions. This is another reason for the choice of $\Omega_N$ as approximating domains (see Remark \ref{remark halfspace}).
	
	\item Finally, we use the arguments of \cite{FG} to obtain a Morse index bound of the solutions obtained. Roughly speaking, since our solutions come from a mountain pass argument, their Morse index is at most 1. This will be used in an essential way in the proof of Theorem \ref{teo:all}.

\end{itemize}

With Theorem \ref{teo:almost} in hand, one could ask whether we can pass to the limit and obtain a nontrivial solution for all values of $c \in (0,\sqrt{2})$. This is relatively easy, see Proposition \ref{ascoli}. The problem here is to show that the limit solution has finite energy.  Let us point out that the boundedness of the energy cannot be deduced only by using Pohozaev-type identities, and more delicate arguments are needed. We give two results on this aspect.

The only requirement of the next theorem is $d=3$:

\begin{thm} \label{teo:3} Assume that $d=3$. Let $c \in (0,\sqrt{2})$, $c_n \in E$, $c_n \to c$, where $E$ is the set given by Theorem \ref{teo:almost}. Let $\psi_n$ be the finite energy solutions with speed $c_n$ given by that theorem. Then there exists $\xi_n \in \R^d$ such that:

$$ \psi_n(\cdot - \xi_n) \to \psi \ \mbox{ in } C^k_{loc}(\R^d),$$

where $\psi$ is a nontrivial finite energy solution of \eqref{eq:GPellip} with speed $c$.

\end{thm}

Observe that Theorems \ref{teo:almost} and \ref{teo:3} give an alternative proof of the result of Mari\c s \cite{Ma} for the Gross-Pitaevskii equation.

\medskip 

Under minor changes, Theorems \ref{teo:almost} and \ref{teo:3} can be adapted to $d \geq 4$: the problem there is the fact that the term $(1-|\psi|^2)^2$ becomes critical or supercritical with respect to the Sobolev embedding. However, since this term has a positive sign in the functional, this issue could be fixed by changing suitably the functional setting, or, alternatively, by using a convenient truncation argument. For the sake of brevity we will not do so and restrict ourselves to the relevant spatial dimensions $d=2$ or $d=3$.

\medskip Regarding compactness of solutions, the case $d=2$ is, again, more involved. It presents analytical difficulties and also topological obstructions, see Remarks \ref{r1}, \ref{r2}. In dimension 2 we are able to conclude only under some assumptions on the vortex set of the solutions:

\begin{thm} \label{teo:all} Take $c \in (0,\ \sqrt{2})$, $c_n \in E$ with $c_n \to c$ and $\psi_n$ the finite energy solutions with speed $c_n$ given by Theorem \ref{teo:almost}. Assume that 
	
\begin{enumerate}
	\item either $\psi_n$ are vortexless, that is, $\psi_n(x) \neq 0 $ for all $x \in \R^d$,
	\item or there exists $R>0$, $\delta>0$ such that:
	
	\begin{equation} \label{control}  \{x \in \R^d: \ \psi_n(x)=0 \} \subset B(0, R) \mbox{ and } |\psi_n(x)| \geq \delta \ \forall \, x \in \partial B(0,R). \end{equation}
	
\end{enumerate}

Then there exists $\xi_n \in \R^d$ such that:

$$ \psi_n(\cdot - \xi_n) \to \psi \ \mbox{ in } C^k_{loc}(\R^d),$$

where $\psi$ is a nontrivial finite energy solution of \eqref{eq:GPellip} with speed $c$.

\end{thm}

The proofs of both Theorem \ref{teo:3} and Theorem \ref{teo:all} follow similar ideas, which include the following:

\begin{itemize}
	
	\item A fundamental tool is the use of a lifting, that is, the existence of real functions $\rho_n(x)$, $\theta_n(x)$ such that $\psi_n = \rho_n e^{i \theta_n}$. This is always possible if the solutions are vortexless. If the solutions present vortices, one needs some information on the location of the vortex set. In Theorem \ref{teo:3} one can show that the vortices are included in a set of disjoint balls, and that the number of balls and their radius is bounded. Generally speaking, a nonvanishing function $\psi$ admits a lifting if its domain is simply connected. Since the complement of a disjoint union of closed balls is simply connected if $d=3$, we can find a lifting outside those balls. In dimension 2 this is no longer true, though, and we can use a lifting only in the complement of one ball, since the total degree of a finite energy solution is 0 (see \cite{gravejat-AIHP}). 
	
	\item We reason by contradiction assuming that $\mathcal{E}(\psi_n) \to +\infty $. A Pohozaev-type identity implies that $$ \sum_{k=2}^d \int_{\R^d}  |\partial_{x_k} \psi_n|^2 = (d-1) I^{c_n}(\psi_n),$$ and $I^{c_n}(\psi_n)$ is bounded by Theorem \ref{teo:almost}. In our arguments we can pass to a limit (locally) which is a 1-D solution of the Gross-Pitaevskii equation (with finite or infinite energy). The knowledge of those 1-D solutions is essential at this point. For instance, in the proof of Theorem \ref{teo:all} we are able to obtain in the limit a circular solution $\psi(x_1) = \rho_0 e^{i \omega_0 x_1}$, with  $\rho_0^2 < \frac{2}{3} (1+c^2/4)$. But it turns out that such solution has infinite Morse index, and we reach a contradicion.
	
\end{itemize}

Under minor changes, it is possible to adapt the results of this paper to an equation with more general nonlinearities, namely:

$$ i c\partial_{x_1}\psi +\Delta \psi+ F(|\psi|)\psi=0  \ \  \text{ on } \R^d.$$

Several assumptions on the nonlinearity $F$ would be in order. However, for the sake of brevity and clarity, we have preferred to focus on the prototype model of the Gross-Pitaevskii equation in this paper.

The rest of the paper is organized as follows. Section 2 is devoted to the setting of the notation and some preliminary results. In Section 3 we begin the proof of Theorem \ref{teo:almost} by considering problem \eqref{eq:GPstrip} from a variational point of view. A main issue here is that we are not able to show that (PS) sequences have bounded energy. This problem is solved for almost all values of $c$ via the monotonicity trick of Struwe in Section 4. We are able to find sequences of slabs $\Omega_{k(N)}$ for which those solutions have uniformly bounded energy. In Section 5 we pass to the limit avoiding vanishing or concentration on the boundary, concluding the proof of Theorem \ref{teo:almost}. 
Sections 6 and 7 are devoted to the proofs of Theorems \ref{teo:3}, \ref{teo:all}, respectively. The appendix deals with the Morse index computation of the 1-D circular solutions of the Gross-Pitaevskii equation, which is needed in the conclusion of Theorem \ref{teo:all}.

\bigskip

\medskip {\bf Acknowledgements:} The authors wish to thank Rafael Ortega for many discussions on the 1-D solutions of the Gross-Pitaevskii equation, and also for his help in the elaboration of the Appendix.

\section{Preliminaries}

In this section we collect some well-known properties of solutions of the Gross-Pitaevskii equation. We begin by stablishing the notation that we will use throughout the paper.

\medskip

{\bf Notation:} We denote by $\lan z_1, z_2 \ran$ the real scalar product of two elements in $\mathbb{C}$, that is, $\lan z_1, z_2 \ran = Re (z_1 \overline{z_2})$. We denote instead by $\xi_1 \cdot \xi_2$ the real scalar product in $\R^d$, to avoid confusion.

We shall use the letter $\psi$ for complex valued functions, and we will denote its real and imaginary part by $u$ and $v$, respectively, so that $\psi = u + i v$. Moreover, we will write $\rho$ to denote its modulus, that is, $\rho^2 = u^2 + v^2= \lan \psi, \overline{\psi} \ran$.

We denote the partial derivatives by $\partial_{x_1} \psi$, but sometimes we will use $\psi_{x_1}$ for convenience.

\medskip

In next lemma we are concerned with the regularity of solutions and the uniform boundedness of their derivatives.

\begin{lem} \label{lem:bound} Any solution $\psi$ of \eqref{eq:GPellip} or \eqref{eq:GPstrip} is of class $C^{\infty}$ and, for any $k \in \N$, there exists $C_k>0$ such that $\| D^k \psi (x) \| \leq C_k$ for any $x \in \R^d$.

\end{lem}

The above result is well-known. The starting point is the $L^\infty$ estimate:

$$ \| \psi \|_{L^\infty} \leq \sqrt{1+c^2/4}.$$

This was proved in \cite{farina} for all entire solutions of \eqref{eq:GPellip} (not only those with finite energy). The argument works equally well for problem \eqref{eq:GPstrip} since the boundary condition is compatible with the $L^\infty$ bound. From this, one can obtain the result via local elliptic regularity estimates.

Indeed the solutions are analytic, see \cite{bgs-cmp}[Theorem 2.1] for more details.

\medskip 

Next lemma gives a Pohozaev identity:

\begin{lem} \label{lem:poho} 

Let $\p$ be a finite energy solution of \eqref{eq:GPellip}. Then:

$$\frac{d-2}{2} \int_{\R^d} |\nabla \psi|^2  -(d-1) c \mathcal{P}(\psi)+ \frac d 4  \int_{\R^d} \left(1-|\psi|^2\right)^2 =0.$$

\end{lem}

\begin{proof}

See for instance \cite{gravejat-CMP}, or  \cite{bgs-cmp}[Lemma 2.5 and following].

\end{proof}

Next identity is also of Pohozaev-type, but only uses the invariance of the domain by dilations in the $\tilde{x}$ variable:

\begin{lem} \label{lem:poho2}  Let $\p$ be a finite energy solution of either \eqref{eq:GPellip} or \eqref{eq:GPstrip}. Then the following identity holds:
	$$(d-3) A(\psi)+(d-1)B(\psi)=0,$$
	where
	\[
	A(\psi)=\frac 12 \sum_{j=2}^d  \int  |\nabla_{x_j} \psi|^2  
	\]
	and
	\[
	B(\psi)=\frac{1}{2} \int |\partial_{x_1} \psi|^2 + \frac 14  \int \left(1-|\psi|^2\right)^2 -c\mathcal{P}(\psi). 
	\]
	Moreover, by the definition of the Lagrangian \eqref{eq:ac}, we conclude that 
	\begin{equation}\label{eq:A bounded} I(\psi)=\frac{2}{d-1}A(\psi)\geq 0. \end{equation}
	Finally, $I(\psi)=0$ if and only if $\psi$ is a constant function of modulus 1.

\end{lem}

\begin{proof}
 The case of \eqref{eq:GPellip} has actually been proved in \cite{gravejat-CMP}[Proposition 5], taking into account \cite{bgs-cmp}[Lemma 2.5] (see also\cite{Maris-SIAM}[Proposition 4.1]). The case of the domain $\Omega_N$ is completely analogous and is based on the fact that the dilations $(x_1, \tilde{x}) \mapsto (x_1, \lambda \tilde{x} )$ leave the domain $\Omega_N$ invariant.
\end{proof}

The following decay estimate has been proved in \cite{gravejat-AIHP}:

\begin{lem} \label{lem:decay} Let $\p$ be a finite energy solution of \eqref{eq:GPellip} satisfying \eqref{limit1}. Then the following asymptotics hold:

$$ |v (x) | \leq \frac{K}{1+|x|^{d-1}}, \ \ |u(x)- 1 | \leq \frac{K}{1+|x|^{d}},$$

$$ |\nabla v (x) | \leq \frac{K}{1+|x|^{d}}, \ \ |\nabla u (x) | \leq \frac{K}{1+|x|^{d+1}}.$$

Outside a ball $B(0,R)$ containing all vortices, $\p$ can be lifted as $\psi = \rho e^{i\theta}$. Then the above decay estimates can be written as:

$$ |\theta (x) | \leq \frac{K}{1+|x|^{d-1}}, \ \ | \rho(x) - 1 | \leq \frac{K}{1+|x|^{d}},$$

$$ |\nabla \theta (x) | \leq \frac{K}{1+|x|^{d}}, \ \ |\nabla \rho (x) | \leq \frac{K}{1+|x|^{d+1}}.$$

In particular, the definition \eqref{eq:MO} of the momentum is well defined for any finite energy solution of \eqref{eq:GPellip}.
	
\end{lem}

We now define the Morse index of a solution of \eqref{eq:GPellip}:

\begin{definition} \label{Morse} Let $\psi$ be a solution of \eqref{eq:GPellip} (either with finite or infinite energy). We define its Morse index $ind(\psi)$ as:
	
	$$ \sup \{dim\, Y: \ Y \subset C_0^{\infty}(\R^d) \mbox{ vector space,  } Q(\phi)<0 \ \forall \, \phi \in Y \},$$
where
\begin{equation} \label{defQ} Q(\phi)=\int_{\R^d} |\nabla \phi|^2  - c \lan \phi, i \partial_{x_1} \phi \ran - (1 - |\psi|^2) |\phi|^2 +2 \lan \phi, \psi \ran^2. \end{equation}

If that set is not bounded from above, we will say that its Morse index is $+\infty$.

\medskip Observe that, at least formally, $Q(\phi) = I_c''(\psi)[\phi, \phi]$, and hence the Morse index is nothing but the maximal dimension for which $I_c''(\psi)$ is negative definite.

\begin{remark} \label{remark morse} An useful property of the so-defined Morse index is that it is decreasing under convergence in compact sets. Being more specific, assume that $\psi_n$ is a sequence of solutions of \eqref{eq:GPellip} or \eqref{eq:GPstrip}. Assume also that $ind(\psi_n) \leq m$ and $\psi_n$ converges to $\psi_0$ in $C^1_{loc}$ sense. Then $ind (\psi_0) \leq m$. 
	
This property will be essential, in particular, in the proof of Theorem \ref{teo:all}.
\end{remark}

\end{definition}

\section{The variational approach of Problem \eqref{eq:GPstrip}}

We first recall the definition of $\Omega_N$ \eqref{omegaN} and observe that in the Sobolev Space  $H_0^1(\Omega_N)$ the Poincar\'{e} inequality holds:
\begin{equation} \label{eq:poincare}
\int_{\Omega_N} |\phi|^2  \leq C_N \int_{\Omega_N} |\nabla \phi|^2 \ \ \forall \ \phi \in H_0^1(\Omega_N).
\end{equation}

If we combine this with the Sobolev inequality we obtain that

\begin{equation} \label{sob}
\| \phi \|_{L^p} \leq C_N \| \nabla \phi \|_{L^2}, \ \  \left \{ \begin{array}{ll}  p\in [2, 6] & \mbox{if } d=3, \\ \\ p \geq 2 & \mbox{if } d=2. \end{array}\right.
\end{equation} 

Let us define the action functional $I^c_N$ as the Lagrangian $I^c$ defined in \eqref{eq:ac} restricted to the affine space $ 1+H_0^1(\Omega_N)$, that is,
$$
I^c_N(\psi):= \mathcal{E}(\psi) - c \mathcal{P}(\psi)= \frac{1}{2} \int_{\Omega_N} |\nabla \psi|^2  -c\mathcal{P}(\psi)+ \frac 14  \int_{\Omega_N} \left(1-|\psi|^2\right)^2 
$$

%Hereafter we denote $\tilde H_0^1(\Omega_N)= 1+H_0^1(\Omega_N)$ and
%$$||\psi||_{\tilde H^1(\Omega_N)}=||\psi-1||_{ H^1_0(\Omega_N)}.$$ 

We notice that thanks to the identities
$$\mathcal{P}(u+i v)=- \int_{\Omega_N} (u(x)-1)\partial_{x_1} v(x),$$
$$(1-u^2-v^2)^2=(2(u-1)+(u-1)^2+v^2)^2$$
%and the interpolation inequalities
%\begin{equation}\label{eq:interp}
%\left\{ \begin{aligned}
%||u-1||_{L^3(\Omega_N)}^3 \leq C ||u-1||_{L^2(\Omega_N)}^2||\nabla u||_{L^2(\Omega_N)}\leq C (4N^2) ||\nabla u||_{L^2(\Omega_N)}^3 \ \text{ if } d=2 \\
%||u-1||_{L^4(\Omega_N)}^4 \leq C ||u-1||_{L^2(\Omega_N)}^2||\nabla u||_{L^2(\Omega_N)}^2 \leq C (4N^2) ||\nabla u||_{L^2(\Omega_N)}^4  \ \text{ if } d=2\\
%||u-1||_{L^3(\Omega_N)}^3 \leq C ||u-1||_{L^2(\Omega_N)}^{\frac 32}||\nabla u||_{L^2(\Omega_N)}^{\frac 32 }\leq C (2N)^{\frac 32} ||\nabla u||_{L^2(\Omega_N)}^4    \ \text{ if } d=3\\
%||u-1||_{L^4(\Omega_N)}^4 \leq C ||u-1||_{L^2(\Omega_N)}||\nabla u||_{L^2(\Omega_N)}^3 \leq C (2N) ||\nabla u||_{L^2(\Omega_N)}^4   \ \text{ if } d=3
%\end{aligned}\right..
%\end{equation}
the action functional $I^c_N$ is $C^2$ in $H^1_0(\Omega_N)$. Our aim is to prove the existence of critical points of the action functional where the velocity parameter
if fixed; these critical points correspond to solution to \eqref{eq:GPstrip}. Let us point out that $H_0^1(\Omega_N)$ is included in $H_0^1(\Omega_{N'})$ if $N'>N$ (up to extension by $0$).

Our strategy is to prove that $I^c_N$ has a mountain pass geometry on $1+ H_0^1(\Omega_N)$.  More precisely we aim to prove that 
\begin{equation}\label{gamma}
\gamma_N(c):= \inf_{g \in \Gamma} \max_{t\in [0,1]}I^c_N(g(t)) > 0,
\end{equation}
where
\begin{equation}\label{Gamma}
\Gamma(N) =\{g \in C([0,1], (1+H_0^1(\Omega_N)): \ g(0)=1, \ g(1)=\psi_0\},
\end{equation}
where $\psi_0$ is chosen so that $I^c_N(\psi_0)<0$.
\begin{prop} \label{min-max}
Given any $c_0 \in (0, \sqrt{2})$, there exist $N_0>0$, $\psi_0 \in 1 + H_0^1(\Omega_{N_0})$ and $\chi(c_0)>0$ such that $\forall N \geq N_0$, $c \in [c_0, \sqrt{2})$:

\begin{enumerate}
	\item[a)] $I^c_N(\psi_0)<0$.
	\item[b)] $0 < \gamma_N(c) \leq \chi(c_0)$.
\end{enumerate}

\end{prop}
\begin{proof}

We can write the action functional $I^c_N(\psi)$, as:
$$ I^c_N(\psi) = \int_{\Omega_N} \frac{1}{2} |\nabla u|^2 + \frac{1}{2}  |\nabla v|^2 -c (1-u)\partial_{x_1} v +\frac 1 4  (2(u-1)+(u-1)^2+v^2)^2.
$$
Moreover we have the elementary inequality $c x y\leq \frac{c^2}{4}x^2+y^2$, so that
 
\begin{eqnarray*} 
I^c_N(\psi)  \geq  \displaystyle \int_{\Omega_N} \frac{1}{2}  |\nabla u|^2 + \left(\frac{1}{2} -\frac{c^2}{4} \right)  |\nabla v|^2 - (u-1)^2 + \frac{(2(u-1)+(u-1)^2+v^2)^2}{4} \\ 
\geq \displaystyle\int_{\Omega_N} \frac{1}{2}  |\nabla u|^2 + \left(\frac{1}{2} -\frac{c^2}{4} \right)  |\nabla v|^2 - |u-1|^3 - |u-1|v^2. \qquad \qquad \qquad \quad \ \ \ 
\end{eqnarray*}
By using Holder inequality and \eqref{sob}, we obtain:
$$
I^c_N(\psi) \geq  \left(\frac{1}{2} -\frac{c^2}{4} \right) ||\psi-1||_{H^1_0(\Omega_N)}^2 - K||\psi-1||_{H^1_0(\Omega_N)}^3,
$$
and hence $\psi=1$ is a local minimum of the action functional whenever $c^2<2$.

In \cite{Ma}, Lemma 4.4, a compactly supported function $\phi_0$ is found so that $I^{c_0}(1+\phi_0)<0$. So it suffices to take sufficiently large $N_0$ such that $\Omega_N \supset supp \, \psi_0$, to obtain a).

Finally, define $\gamma_0(t) = 1 + t \phi_0$, which obviously belongs to $\Gamma(N)$ for all $N \geq N_0$. Observe that:

$$ I_N^c(\gamma_0(t)) = \mathcal{E}(\gamma_0(t)) - c \, t^2 \mathcal{P}(\psi_0).$$

As commented above $I_N^c(\psi_0)<0$, which implies that $\mathcal{P}(\psi_0)>0$. Hence, for all $c \geq c_0$,

$$ I_{N}^c(\gamma_0(t)) \leq I_N^{c_0}(\gamma_0(t)) \leq \max_{t \in [0,1]} I_{N_0}^{c_0} \circ \gamma_0(t)=\chi(c_0),$$
by definition. As a consequence, $\gamma_N(c) \leq \chi(c_0)$ for all $N \geq N_0$, $c \geq c_0$.
\end{proof}

It is standard (see for instance \cite{AM, willem}) that the mountain pass geometry induces the existence of a Palais-Smale sequence at the level $\gamma_N$. Namely, a sequence $\psi_n$ such that
$$I_N^c(\psi_n)=\gamma_N(c)+o(1), \ \ \ ||(I_N^c)'(\psi_n)||_{H^{-1}_0(\Omega_N)}=o(1).$$

It is not clear if such Palais-Smale sequences are bounded or not; this is one of the main difficulties. The question of the existence of Palais-Smale sequences with bounded energy for almost all values of $c$ will be addressed in next section. In what follows we show that, if bounded, such sequences give rise to critical points of $I^c_N$. 
\begin{lem}\label{eq:novan}
Let $d=2,3$ and $\left\{Q_j\right\}$ be the set of disjoint unitary cubes that covers $\Omega_N$. If $\psi_n=u_n+ i v_n$ is a bounded vanishing sequence in $1+ H^1_0(\Omega_n)$, i.e. such that
$$\sup_j \int_{Q_j} |u_n-1|^p+|v_n|^p  \rightarrow 0$$
for some $2\leq p<\infty$ if $d=2$, $2\leq p<6$ if $d=3$, then
$$\int_{\Omega_N} |u_n-1|^r+|v_n|^r  \rightarrow 0$$
for any $2< r<\infty$ if $d=2$, $2<r<6$ if $d=3.$
\end{lem}
\begin{proof}
The proof is standard, see e.g \cite{lions2} Lemma I.1.
%and we show it only for reader convenience. Let us consider for simplicity the case $N=2$ and let $p<s.$ Now we can choose $q>s$ such that $q(s-p)>2(q-p)$.
%By interpolation we have
%$$\int_{\Omega_N} |u_n-1|^s dx=\sum_j\left( \int_{Q_j} |u_n-1|^s dx\right)\leq $$
%$$\leq  \sum_j  ||u_n-1||_{L^p(Q_j)}^{s\alpha}||u_n-1||_{L^q(Q_j)}^{s(1-\alpha)}\leq \left(\sup_j ||u_n-1||_{L^p(Q_j)}^{s\alpha} \right)\sum_j ||u_n-1||_{L^q(Q_j)}^{s(1-\alpha)}$$
%where $\alpha=\frac{p(q-s)}{s(q-p)}$.
%Now we have 
%$$\sum_j ||u_n-1||_{L^q(Q_j)}^{s(1-\alpha)}\leq C  \sum_j \left( \int_{Q_j} |u_n-1|^2+|\nabla u_n|^2 dx\right)^{\frac{s(1-\alpha)}{2}}$$
%and therefore if $\frac{s(1-\alpha)}{2}>1$ we get
%$$\int_{\Omega_N} |u_n-1|^s dx\leq  C\left(\sup_j ||u_n-1||_{L^p(Q_j)}^{s\alpha} \right) ||u_n||_{\tilde H_0^1(\Omega_N)}^{s(1-\alpha)}\leq C \left(\sup_j ||u_n-1||_{L^p(Q_j)}^{s\alpha} \right).$$
%The condition  $\frac{s(1-\alpha)}{2}>1$ is equivalent to the previous assumption $q(s-p)>2(q-p)$.
%Under the initial  assumption that $\sup_j \sup_j \int_{Q_j} |u_n-1|^p \rightarrow 0$, we conclude that  $||(u_n-1)||_{L^s(\Omega_N)}\rightarrow 0$. This proves that
%$||(u_n-1)||_{L^r(\Omega_N)}\rightarrow 0$ if $r>p$. The case $r\leq p$ follows from the interpolation 
%$||u||_{L^r(\Omega_N)}\leq ||u||_{L^2(\Omega_N)}^{1-\lambda}||u||_{L^s(\Omega_N)}^{\lambda}$.  An analogous estimate  works for $v_n$ and in dimension $d=3$.
\end{proof}
\begin{prop}[Splitting property]\label{prop:12}
Given $0 <c<\sqrt{2}$ and $\psi_n$ a bounded Palais-Smale sequence at the energy level $\gamma_N(c)$. Then there exist $k$ sequences of points $\left\{y_n^j\right\} \subset \{0\} \times \R^{d-1}$, $1\leq j\leq k,$ with $|y_n^j - y_n^k| \to +\infty$ if $j \neq k$, such that, up to subsequence,
\begin{eqnarray}
\psi_n -1  =w_n + \sum_{j=1}^k (\psi^j(\cdot + y_n^j)-1)  \text{ with } w_n \rightarrow 0 \text { in }  H^1_{0}(\Omega_N), \\ 
||\psi_n-1||_{ H^1_0(\Omega_N)}^2\rightarrow \sum_{j=1}^k  ||\psi^j-1||_{H^1_0(\Omega_N)}^2, \qquad \qquad \qquad \label{eq:spl12} \\ 
I^c_N(\psi_n ) \rightarrow \sum_{j=1}^k I^c_N(\psi^j), \qquad \qquad \qquad \qquad \quad \label{eq:spl3}
\end{eqnarray}
where $\psi^j$ are nontrivial finite energy solutions to \eqref{eq:GPstrip}. 

In particular $I^c_N(\psi^j)\leq \gamma_N(c) \leq \chi(c_0)$, $\mE({\psi^j}) \leq \limsup_{n \to +\infty}  \, \mE(\psi_n)$ for all $j =1, \dots k$.
\end{prop}
\begin{proof}
In this proof, for the sake of clarity, we drop the dependence on $c$. 

We first claim that $\psi_n$ is not vanishing. Reasoning by contradiction, by means of Lemma \ref{eq:novan} we have
\begin{equation}\label{eq:consvan}
\int_{\Omega_N} |u_n-1|^r+|v_n|^r  \rightarrow 0
\end{equation}
for any $2< r<\infty$ if $d=2$, $2<r<6$ if $d=3.$\\
Then, 
$$\int_{\Omega_N} (1-u_n^2-v_n^2)^2=\int_{\Omega_N} 4(u_n-1)^2+(u_n-1)^4+v_n^4+$$
$$+\int_{\Omega_N} 4(u_n-1)^3+4(u_n-1)v_n^2+2(u_n-1)^2v_n^2$$
and it follows by H\"older and \eqref{eq:consvan}  that
\begin{equation}\label{eq:crucvan}
\int_{\Omega_N} (1-u_n^2-v_n^2)^2 =\int_{\Omega_N} 4(u_n-1)^2 +o(1).
\end{equation}
Therefore
\begin{eqnarray} 
&I_N(u_n,v_n)=\frac{1}{2} \displaystyle \int_{\Omega_N} |\nabla u_n|^2 + \frac{1}{2} \int_{\Omega_N} |\nabla v_n|^2 -c \int_{\Omega_N} (1-u_n(x))\partial_{x_1} v_n(x) + \nonumber\\
&+ \displaystyle \int_{\Omega_N} (u_n-1)^2 +o(1) \label{eq:vanen}
\end{eqnarray}
On the other hand direct computation gives
$$o(1)=I_N'[\psi_n](1-\psi_n)=\int_{\Omega_N}|\nabla u_n|^2+|\nabla v|^2  -2c \int_{\Omega_N} (1-u_n)\partial_{x_1} v_n(x) -\\$$
$$+\int_{\Omega_N} (1-u_n^2-v_n^2)(u_n(1-u_n)-v_n^2).$$
Arguing as before we notice that
$$\int_{\Omega_N} (1-u_n^2-v_n^2)v_n^2=o(1)$$
and, thanks to  H\"older inequality and \eqref{eq:crucvan}
$$\int_{\Omega_N} (1-u_n^2-v_n^2)(u_n(1-u_n))=\int_{\Omega_N} (1-u_n^2-v_n^2)(u_n-1+1)(1-u_n))$$
$$=\int_{\Omega_N} (1-u_n^2-v_n^2)(1-u_n)+o(1)= \int_{\Omega_N} (1-u_n^2)(1-u_n)+o(1) = $$
$$ \int_{\Omega_N} (2(1-u_n)-(1-u_n)^2)(1-u_n)+o(1)= 2||u_n-1||_{L^2(\Omega_N)}^2+o(1).$$
We get hence that
\begin{equation}\label{eq:vander}
I_N'[\psi_n](1-\psi_n)=\int_{\Omega_N}|\nabla u_n|^2+|\nabla v_n|^2  -2c \int_{\Omega_N} (1-u_n)\partial_{x_1} v_n(x) +2||u_n-1||_{L^2(\Omega_N)}^2+o(1)
\end{equation}
Taken into account  \eqref{eq:vanen} and\eqref{eq:vander} we conclude
$$\gamma(N)+o(1)=I_N(\psi_n)-\frac 12 I_N'[\psi_n](1-\psi_n)=o(1),$$
a contradiction. 

\medskip 

Once vanishing is excluded, there exists a sequence  $y^1_n \in \{0\} \times \R^{d-1}$ and $\psi^1 \in 1 + H_0^1(\Omega_N)$, $\psi^1 \neq 1$,  such that
$$\psi_n(\cdot + y^1_n) - \psi^1 \rightharpoonup 0 \mbox { in } H_0^1(\Omega_N)$$
up to a subsequence. From the definition of weak convergence we obtain that
$$\frac 12 \int_{\Omega_N}|\nabla \psi_n|^2 -c \mathcal{P}(\psi_n)=\frac 12 \int_{\Omega_N}|\nabla \psi^1|^2 -c \mathcal{P}(\psi^1)+$$
$$+\frac 12 \int_{\Omega_N}|\nabla (\psi_n-\psi_0)|^2 -c \mathcal{P}(\psi_n-\psi^1)+o(1).$$
Now we notice that the nonlinear term fulfills the following splitting property 
\[
\int_{\Omega_N} \left(1-|\psi_n|^2\right)^2 =\int_{\Omega_N} \left(1-|\psi^1|^2\right)^2 +\int_{\Omega_N} \left(1-|\psi_n-\psi^1|^2\right)^2 +o(1).
\]
The proof of the splitting property is standard.
A a consequence the action splits as 
\begin{equation}\label{eq:split}
I_N(\psi_n ) = I_N(\psi^1)+ I_N((\psi_n-\psi^1))+o(1).
\end{equation}

Clearly $\psi^1$ is a weak solution of \eqref{eq:GPstrip}. Now if $\psi_n-\psi^1 \rightarrow 0$ in $\tilde H^1_0(\Omega_N)$ the lemma is proved. Let us assume the contrary, i.e. that  
$z_n^1=\psi_n-\psi^1 \rightharpoonup 0$ and
$z_n^1=\psi_n-\psi^1 \nrightarrow 0$ in $H^1_0(\Omega_N)$. 

We aim to prove that there exists a sequence of points $y^2_n$ and $\psi^2 \in 1 + H_0^1(\Omega_N)$, $\psi^2 \neq 1$, such that
$z_n^1(\cdot +y^2_n) - \psi^2 \rightharpoonup 0$.  Let us argue again by contradiction assuming that the sequence $z_n^1$ vanishes which means 
by Lemma \ref{eq:novan} that
\begin{equation*}
\int_{\Omega_N} |u_n-u^1|^r+|v_n-v^1|^r  \rightarrow 0
\end{equation*}
for any $2< r<\infty$ if $d=2$, $2<r<6$ if $d=3,$
where $\psi^1 = u^1 + i v^1$.
We have
$$I_N'[\psi_n](1-\psi_n)=\int_{\Omega_N}|\nabla u_n|^2+|\nabla v_n|^2  -2c \int_{\Omega_N} (1-u_n)\partial_{x_1} v_n(x) +\\$$
$$+\int_{\Omega_N} (1-u_n^2-v_n^2)(u_n(1-u_n)-v_n^2)=o(1),$$
and
$$I_N'[\psi^1](1-\psi^1)=\int_{\Omega_N}|\nabla u^1|^2+|\nabla v^1|^2  -2c \int_{\Omega_N} (1-u^1)\partial_{x_1} v^1(x) -\\$$
$$+\int_{\Omega_N} (1-(u^1)^2-(v^1)^2)(u^1(1-u^1)-(v^1)^2)=0.$$
Using the splitting property \eqref{eq:split} and using  $I_N'[\psi_n](1-\psi_n)-I_N'[\psi^1](1-\psi^1)=o(1)$ we get
$$\int_{\Omega_N}|\nabla z_n^1|^2 -2c \int_{\Omega_N} Re(z_n^1)\partial_{x_1}Im( z_n^1(x)) +o(1)=$$
$$\underbrace{\int_{\Omega_N} (1-u_n^2-v_n^2)(u_n(u_n-1)+v_n^2)-\int_{\Omega_N} (1-(u^1)^2-(v^1)^2)(u^1((u^1)-1)+(v^1)^2)}_{=I_4}.$$
Now, using the elementary inequality $c x y\leq \frac{c^2}{4}x^2+y^2$ we have
\begin{equation}\label{eq:import2}
\int_{\Omega_N}|\nabla Re(z_n^1)|^2+ (1-\frac{c^2}{2})\int_{\Omega_N}|\nabla Im(z_n^1)|^2 \leq 2 \int_{\Omega_N}|Re(z_n^1)|^2+ I_4+o(1).
\end{equation}
Notice that
$$(1-u_n^2-v_n^2)(u_n(1-u_n)-v_n^2)=(1-u_n^2-v_n^2)^2+(1-u_n^2-v_n^2)(u_n-1)$$
and that
$$(1-u_n^2-v_n^2)^2=4(u_n-1)^2+(u_n-1)^4+v_n^4+ 4(u_n-1)^3+4(u_n-1)v_n^2+2(u_n-1)^2v_n^2.$$
On the other hand
$$(1-u_n^2-v_n^2)(u_n-1)=-2(1-u_n)^2+(1-u_n)^3+v_n^2(1-u_n).$$
Therefore, assuming that $z_n^1=\psi_n-\psi^1\rightharpoonup 0$
we get
$$2 \int_{\Omega_N}|Re(z_n^1)|^2+I_4=o(1)$$
and hence we get a contradiction with \eqref{eq:import2}.\\
%Now using the fact that $z_n=\psi_n-\psi^0\rightharpoonup 0$ we have, after a lenghtly computation, 
%\begin{eqnarray}
%2 \int_{\Omega_N}|Re(z_n)|^2dx+ I_4 \lesssim \int_{\Omega_N}|u_n-u^0|^2dx+\int_{\Omega_N}|u_n-u^0|^3dx+ \label{eq:import2}\\
%+\int_{\Omega_N}|u_n-u^0|^4dx+\int_{\Omega_N}|u_n-u^0|^2|v_n-v^0|^2dx+o(1). \nonumber
%\end{eqnarray}
%Calling $\left\{Q_j\right\}$ be the set of disjoint cubes that covers $\Omega_N$ of edge $L$. We have that
%$$\int_{\Omega_N} |u_n-u^0|^2 dx\leq \sum_j\left( \left(\int_{Q_j} |u_n-u^0|^s dx\right)^{\frac{2}{s}}L^{d(\frac{s-2}{s})}\right)\leq $$
%$$\leq  C L^{d(\frac{s-2}{s})}  ||u_n-u^0||_{\tilde H^1_0(\Omega_N)}^2$$
%and we choose $L$ sufficently small such that $C L^{d(\frac{s-2}{s})}  <\frac 12 (1-\frac{c^2}{2}).$
%Now, as in Lemma \ref{eq:novan}, if $ \sup_j \int_{Q_j} |u_n-1|^p+|v_n|^p dx \rightarrow 0$ then
%$$\int_{\Omega_N}|u_n-u^0|^3dx +\int_{\Omega_N}|u_n-u^0|^4dx+\int_{\Omega_N}|u_n-u^0|^2|v_n-v^0|^2dx=o(1)$$
%and therefore, assuming that $z_n \nrightarrow 0$ in $\tilde H^1_0(\Omega_N)$, 

We have hence proved the existence of a sequence $y_n^2 \in \{0\} \times \R^{d-1}$ and $\psi^2 \in 1 + H_0^1(\Omega_N)$, $\psi^2 \neq 1$, such that 
$$z_n^1(\cdot +y_n^2)- \psi^2 \rightharpoonup 0.$$

Clearly, $|y_n^1 - y_n^2| \to +\infty$, and $\psi^2$ is also a (weak) solution of \eqref{eq:GPstrip}.

Now we can iterate the splitting argument defining $z_n^2=z_n(x,y+y_n^2)-\psi^2$. We aim to show  that we can have only a finite number of iterative steps. \\
We claim that 
$$\inf_{\psi \in \mathcal{N}}||1-\psi||_{H^1_0(\Omega_N)}>0$$
where
$$\mathcal{N}:=\left\{ \psi \in  1 + H^1_0(\Omega_N), \psi\neq 0, \ I_N'[\psi](1-\psi)=0 \right\}.$$
This allows us to conclude thanks to \eqref{eq:spl12}. In order to prove the claim we notice the identity
\begin{eqnarray}
& \displaystyle \int_{\Omega_N} (1-u_n^2-v_n^2)(u_n(1-u_n)-v_n^2)=\int_{\Omega_N} 2(u-1)^2+3(u-1)^3+ \nonumber\\
&+ \displaystyle \int_{\Omega_N} \left( 3 v^2(u-1)+2(u-1)^2v^2+3(u-1)^3+ (u-1)^4+v^4\right)
\end{eqnarray}
such that, thanks to the inequality
$$ -2c \int_{\Omega_N} (1-u)\partial_{x_1} v(x)  \geq   -\frac{c^2}{2} \int_{\Omega_N} |\nabla v|^2 - 2 \int_{\Omega_N} (u(x)-1)^2 $$
we obtain
\begin{eqnarray}
&0=I_N'[\psi](1-\psi)\geq \displaystyle  \int_{\Omega_N} |\nabla u|^2 +(1-\frac{c^2}{2}) \displaystyle  \int_{\Omega_N} |\nabla v|^2 + \label{eq:kfin}\\
& \displaystyle  \int_{\Omega_N} \left(3(u-1)^3+3 v^2(u-1)+2(u-1)^2v^2+3(u-1)^3+ (u-1)^4+v^4\right) \nonumber. 
\end{eqnarray}
From \eqref{eq:kfin} we get
$$\alpha||\psi -1||_{H^1_0(\Omega_N)}^3+\beta||\psi -1||_{H^1_0(\Omega_N)}^4\geq (1-\frac{c^2}{2})||\psi - 1||_{ H^1_0(\Omega_N)}^2$$
and hence $\inf_{\psi \in \mathcal{N}}||\psi-1||_{ H^1_0(\Omega_N)}>0$.\\
Finally, recall that by \eqref{eq:A bounded}, $I_n(\psi^j )>0$. Now, we have up to space translation
$$\gamma_N+o(1)=I_N(\psi_n)=\sum_j I_N(\psi^j(\cdot +y_n^j)) +o(1) \geq I_N(\psi^j) +o(1)$$
and hence $I_N(\psi^j)\leq \gamma_N.$ 
\end{proof}

\section{Uniformly bounded energy solutions in approximating domains}

In this section we prove shall prove the following result:

\begin{prop}  \label{prop-crucial} There exists a subset $E \subset (0,\sqrt{2})$ of plein measure satisfying that, for any $c \in E$, there exists a subsequence $k: \N \to \N$ strictly increasing such that:
	
	\begin{enumerate} 
		
		\item There exists a nontrivial finite energy solution $\psi_N$ of the problem:
		
		$$
		\begin{array}{rcr} i c\partial_{x_1}\psi_N +\Delta \psi+\left(1-|\psi_N|^2\right)\psi_N & = & 0  \ \  \text{ on } \Omega_{k(N)}, \\ \psi_N & = & 1 \text{ on } \partial \Omega_{k(N)}. \end{array} $$
		
		\item $\mathcal{E}(\psi_N) \leq M$ for some positive constant $M=M(c)$ independent of $N \in \N$.
		
		\item $I^c_{k(N)}(\psi_N) \leq \gamma_{k(N)}(c)$.
		
		\item $ind(\psi_N) \leq 1.$
		
	\end{enumerate} 
\end{prop}

One of the key points here is that in (2) the energy is bounded uniformly in $N$. This will be  essential later when passing to the limit as $N \to +\infty$.

\medskip
In a first subsection we will give an abstract result, which is basically well-known but maybe not in this specific form. Later we will apply that result to prove Proposition \ref{prop-crucial}.

\subsection{Entropy and Morse index bounds}

Entropy bounds on Palais-Smale sequences via monotonicity (also called monotonicity trick argument), is a tool first devised in \cite{struwe} that has been used many times since then, applied to a wide variety of problems. Here we need to adapt this argument to obtain uniform bounds in $N$, for a subsequence $k(N)$. Moreover, we will also use Morse index bounds for Palais-Smale sequences, in the spirit of \cite{FG, FG2}. For the sake of completeness, we state and give a proof of a general result in this subsection.

\begin{prop} \label{trick} Let $X$ be a Banach space and  $A$, $B: X \to \R$ two $C^1$ functionals. Assume that either $A(\p) \geq  0$ or $B(\p) \geq 0 $ for all $\p \in X$. For any $c \in J \subset \R^+_0$, we define $I^c:X \to \R$,
	
	$$I^c(\p) = A(\p) - c B(\p).$$
	
	We assume that there are two points $\p_0, \p_1$ in $X$, such that setting
	$$ \Gamma = \{g \in C([0,1], X),\ g(0) = \psi_0,\ g(1) = \p_1\},$$
	the following strict inequality holds for all $ c \in J$:
	$$\gamma(c) = \displaystyle \inf_{g \in \Gamma} \max_{t \in [0,1]} I^c (g(t)) > \max \{I(\p_0),I(\p_1)\}.$$
	
	Then the following assertions hold true:
	
	\begin{enumerate}
		\item If $B \geq 0$, $\gamma$ is decreasing. If instead $A \geq 0$, then the map $\sigma(c) = \frac{\gamma(c)}{c}$ is decreasing. As a consequence, both the maps $ \gamma, \ \sigma $ are almost everywhere differentiable.
		\item Let $c \in J$, $c>0$, be a point of differentiability of $\gamma$. Then, there exists a sequence $\{\psi_n\}$ such that
		
		 \medskip 
		
		\begin{enumerate}
			\item $I^c(\psi_n) \to \gamma(c)$,
			\item $(I^c)'(\psi_n) \to 0$ in $X^{-1}$, and
			\item $dist (\p_n, G_n) \to 0$, where 
			$$G_n=\{ \psi \in X:\ B(\psi) \leq \ - \gamma'(c) + 1/n, \ A(\psi) \leq \gamma(c) -\gamma'(c) c + \frac 1 n\}.$$
		\end{enumerate}
		
		\item Let us define, for any $\delta>0$, the sets
		
	\begin{equation} \label{FGH} \begin{array}{c} F_{\delta} = \{\psi \in X: \ |I^c(\psi) - \gamma(c)|< 2\delta \}, \\ \\ G_{\delta}= \{\psi \in X: \ B(\psi) < \gamma'(c) + \delta, \ A(\psi) < - c^2 \sigma'(c) + \delta \}, \\ \\
	H_\delta =\{\psi \in F_\e:\ dist(\psi, G_\delta) < 2 \delta \}. \end{array} \end{equation}
		
	Let us assume that $A$ and $B$ are uniformly $C^{2, \alpha}$ functionals in $H_\delta$ for some $\delta>0$. Then in (2) we can choose $\psi_n$ satisfying also that:
		
	\medskip	\begin{enumerate}
			\item[d)] There exists a sequence $\delta_n<0$, $\delta_n \to 0$, such that	$$ \sup \{dim \, Y: \ Y \subset X: I_c''(\psi_n)(\phi,\phi)\leq  \delta_n \| \phi\|^2 \ \forall \ \phi \in Y \} \leq 1.$$

	\end{enumerate}
	\end{enumerate}
	
\end{prop}

\begin{remark}
	Observe that, in general, there exist (PS) sequences for $I^c$ for any $c \in J$; see for instance \cite{AM, willem}. The above proposition shows that, for almost all values $c \in J$, there exist (PS) sequences for $I^c$ that satisfy also condition c). This extra condition c) can be useful in order to show convergence of the (PS) sequence. For instance, if either $A$ or $B$ is coercive, Proposition \ref{trick} implies the existence of bounded (PS) sequences, which is an important information in order to derive convergence. This is the result of \cite{jeanjean}.
	
	Assertion (3) comes from \cite{FG} and gives also a Morse index bound of the (PS) sequence. The only novelty is that we have assumed uniform $C^{2,\alpha}$ regularity on the set $H_\delta$. If $A$ or $B$ is coercive, it suffices to have uniform $C^{2,\alpha}$ estimates on bounded sets.
\end{remark}

\begin{remark} To keep the ideas clear, we have stated the result under a mountain-pass geometric assumption. The same principle holds for other types of min-max arguments. What is essential is that the family $\Gamma$ does not depend on the parameter $c$.

\end{remark}

\begin{proof}
	
	The proof  of (1) is inmediate. Indeed, if $B \geq 0$, $I^c(u)$ is decreasing in $c$. Since the family $\Gamma$ is independent of $c$, we have that $\gamma$ is decreasing. Instead, if $A \geq 0$, then the expression $\frac{I^c(u)}{c}$ is decreasing in $c$, and we conclude.
	
	In any of the two cases, the maps $\gamma$, $\sigma$ are differentiable in a set $E \subset J$ of plein measure.
	
	In order to prove (2), we are largely inspired by \cite{jeanjean}. We first state and prove the following lemma:
	
	\begin{lem} Let $c \in E$, $c>0$, then there exists $g_n \in \Gamma$ such that
		
		\begin{enumerate}
			\item $\max_{t \in [0,1]} I^c (g_n(t)) \to \gamma^c$.
			
			\item There exists $\rho_n >0$, $\rho_n \to 0$ such that for all $t \in [0,1]$ with $I^c(g_n(t)) \geq \gamma^c - \frac 1 n$, we have:
			
			$$  B(g_n(t)) \leq -\gamma'(c) + \rho_n, \ \ \limsup_{n \to +\infty} A(g_n(t)) \leq - c^2 \sigma'(c) + \rho_n.$$

		\end{enumerate}	
		
	\end{lem}
	
	\begin{proof}[Proof of the lemma]
		
		Take $c_n\in J$ an increasing sequence converging to $c$. For any $n \in \N$, there exists $g_n \in \Gamma$ such that $\max_{t \in [0,1]} I^{c_n}(g_n(t)) \leq \gamma (c_n) + |c_n-c|^2$. 
		
		If $B \geq 0$ we have that:
		
		$$  \max_{t \in [0,1]} I^{c}(g_n(t)) \leq  \max_{t \in [0,1]} I^{c_n}(g_n(t)) \leq \gamma (c_n) + |c_n-c|^2 \to \gamma(c).$$
		
		Instead, if $A \geq 0$, 
		
		$$  \max_{t \in [0,1]} I^{c}(g_n(t)) \leq \frac{c}{c_n} \max_{t \in [0,1]} I^{c_n}(g_n(t)) \leq \frac{c}{c_n} (\gamma (c_n) + |c_n-c|^2) \to \gamma(c).$$
		
		We now take $t \in [0,1]$ such that $I^c(g_n(t)) \geq \gamma(c) - |c-c_n|^2$. Then:
		
		$$ B(g_n(t)) = \frac{I^{c_n}(g_n(t)) - I^{c}(g_n(t)) }{c-c_n} $$$$\leq  \frac{ \gamma(c_n) + |c_n-c|^2 - \gamma(c) + |c_n-c|^2 }{c-c_n} \to -\gamma'(c).$$
		
		Moreover, $$\limsup_{n \to +\infty} A(g_n(t)) = \limsup_{n \to +\infty} I^c (g_n(t)) + c B(g_n(t)) \leq \gamma(c) - c \gamma'(c).$$
		
		It suffices then to take $c_n = c- \frac{1}{\sqrt{n}}.$
		
	\end{proof}
	
	Recall now the definitions of $F_\delta$, $G_\delta$ and $H_\delta$ given in \eqref{FGH}. By the previous lemma the set $F_{\delta} \cap G_\delta$ is not empty: indeed, the curves $g_n$ pass through $F_{\delta} \cap G_\delta$ for sufficiently large $n$ . Proposition \ref{trick}, (2) is proved if we show that for any $\delta>0$, $$\inf \{ \| (I^c)'(\psi)\|: \ \psi \in H_{\delta} \}=0.$$
	
	We argue by contradiction, and assume that there exists $\delta>0$ such that $\inf \{ \| (I^c)'(\psi)\|: \ \psi \in H_{\delta} \} \geq \delta >0$. A classical deformation argument shows that there exists $\e>0$, $\eta \in C([0,1] \times X:\ X)$ such that:
	
	\begin{enumerate} 
		\item[i)] $\eta(s, \psi) = \psi$ if $s=0$, $ |I^c(\psi)- \gamma(c)| > 2\e$ or $dist (\psi, G_\delta) >2 \delta$.
		\item[ii)] $I^c(\eta(1, \psi)) \leq \gamma(c) - \e$ for all $\psi \in G_\delta$ with $I^c(\psi) \leq \gamma(c) + \e$.
		\item[iii)] $\eta(s, \cdot)$ is a homeomorphism of $X$.
		\item[iv)] $\|\eta(s,\psi) - \psi \| < \delta$,
		\item[v)] $I^c(\eta(s, \psi)) \leq I^c(\psi)$ for all $\psi \in X$.	
		
	\end{enumerate}
	
	The existence of the above deformation can be found in \cite[Lemma 2.3]{willem}, for instance. Actually our notation is compatible with that reference, setting $S= G_\delta$, and taking $\e = \delta^2/8$, for instance.
	
	\bigskip
	
	We now take $n$ large enough and the curve $\gamma_n$ given by the lemma. If $I^c(g_n(t)) < \gamma(c) - \frac 1 n$, by b), we have that $I^c(\eta (1, g_n(t))) < \gamma(c) - \frac 1 n$. In on the contrary, $I^c(g_n(t)) \geq  \gamma(c) - \frac 1 n$, we can combine the lemma with ii) to conclude that $I^c(\eta(1, g_n(t))) \leq  \gamma(c) - \e$. As a consequence,
	
	$$ \max_t I^c(\eta \circ g_n(t)) <  \gamma(c),$$
	a contradiction.
	
	\medskip For the proof of (3) of Proposition \ref{trick}, we just use Theorem 1.7 of \cite{FG} to our sequence of paths $g_n$. It is important to observe that the uniform $C^{2,\alpha}$ regularity in \cite{FG} is required only in the set $H_\delta$ defined above (see, on that purpose, Lemma 3.7 of \cite{FG}).

\end{proof}

\subsection{Proof of Proposition \ref{prop-crucial}}

A direct application of the above results to our setting, combined with Proposition \ref{prop:12}, yields the existence of finite energy solutions in any domain $\Omega_N$, for almost all values of $c$. The problem here is that the energy of those solutions could diverge if we make $N \to +\infty$. In order to obtain uniform bounds independent of the parameter $N$, we need a more subtle application of Proposition \ref{trick}.

\medskip Define:
\begin{enumerate}
	\item  $X= 1+ H_0^1(\Omega_N)$, which is an affine Banach space, for which Proposition \ref{trick} also holds;
	\item $A(\psi)= \mE(\p)$, which is positive and coercive;
	\item $B(\p)= \mP(\p)$, the momentum;
	\item $J= (c_0, \sqrt{2})$ for a fixed value $c_0>0$. 
\end{enumerate}

For $N\geq N_0$ the functional $I^c_N$ has a min-max geometry (see Proposition \ref{min-max}); recall that $\gamma_N(c)>0$ is the function that associates to a speed $c \in J$ the min-max value of $I^c_N$. Clearly, $\sigma_N(c)= \frac{\gamma_N(c)}{c}$ is decreasing in $c$ as Proposition \ref{trick} shows. By Proposition \ref{trick}, there exists a bounded (PS) sequence in $H^1_0(\Omega_N)$ at level $\gamma_N(c)$. Proposition \ref{prop:12} yields then the existence of a solution $\psi_N$ with: 

$$I_N^c(\psi_N) \leq \gamma_N(c), \ \mathcal{E}(\psi_N)=A(\psi_N) \leq -c^2 \sigma_N'(c).$$

Since $A$ is coercive, here the set $H_\delta$ is uniformly bounded, and $I$ is clearly uniformly $C^{2,\alpha}$ in bounded sets. By Proposition \ref{trick}, 3), we have that:

$$ \sup \{dim \, Y: \ Y \subset H_0^1(\Omega_N): (I_N^c)''(\psi_N)(\phi,\phi) < 0 \ \forall \ \phi \in Y \} \leq 1.$$

We are now concerned with passing to the limit as $N \to +\infty$. In order to control the energy of the solutions $\psi_N$, we reason as follows.  

Recall Proposition \ref{min-max}, b), and that $\sigma_N(c)$ is decreasing in $c$; then, for $N\geq N_0$,

\begin{equation} \label{prima}
\frac{\chi(c_0)}{c_0} \geq  \frac{\gamma_N(c_0)}{c_0} \geq \frac{\gamma_N(c_0)}{c_0} -\frac{\gamma_N(c)}{c} \geq \int_{c_0}^c |\sigma_N'(s)| \, ds .\end{equation}

Let us now define the sets $$D_{N,M}= \{c \in (c_0, \sqrt{2}): \sigma_N \mbox{ is not  differentiable or } |\sigma_N'(c)|> M   \},$$ for all $N$, $M \in \N$, $N \geq N_0$. Clearly the sets $D_{N,M}$ also depend on $c_0$, but we avoid to make that dependence explicit in the notation for the sake of clarity.

By \eqref{prima}, we have that $$|D_{N,M}| \leq \frac{\chi(c_0)}{c_0 M}.$$

The following claim is the key to be able to pass to the limit for enlargins slabs preserving bounded energy.

\medskip {\bf Claim:}  The set $D(c_0)$ defined as:

$$ D(c_0) = \displaystyle \cap_{M \in \N} \cup_{N \geq N_0} \cap_{k \geq N} D_{k,M}$$

has $0$ measure.

\medskip Indeed, the sets $\cap_{k \geq N} D_{k,M}$ are increasing in $N$, and all of them satisfy that have measure smaller than $\frac{\chi(c_0)}{c_0 M}$. Hence the same estimate works also for the union in $N$. Now, $D(c_0)$ is a set given by an intersection of sets of measure $\frac{\chi(c_0)}{c_0 M}$, $M \in \N$, so that $D(c_0)$ has $0$ measure.

\medskip Finally, we can set $$D=  \cup_{n=1}^{+\infty} \, D(1/n),$$ which has also $0$ measure.

\bigskip

Let us define $E= (0,\sqrt{2}) \setminus D$, and take $c \in E$. We can fix $n \in N$ such that $c_0=1/n< c$, and $c \notin D(c_0)$. Then, there exists $M(c)$ and a subsequence $k(N)$ such that $|\sigma_{k(N)}'(c)| \leq M(c)$. By Proposition \ref{trick}, for any of these slabs $\Omega_{k(N)}$ there exists a Palais-Smale sequence with bounded energy. According to Proposition \ref{prop:12}, this gives rise to a solution $\psi_{k(N)} \in 1+ H_0^1(\Omega_{k(N)})$ such that:

$$I_{k(N)}^c(\psi_{k(N)}) \leq \gamma_{k(N)}(c), \ \mE(\psi_{k(N)} ) \leq M(c) c^2.$$

This concludes the proof of Proposition \ref{prop-crucial}.

\section{Proof of Theorem \ref{teo:almost}}

In view of Proposition \ref{prop-crucial}, we aim to conclude the proof of Theorem \ref{teo:almost} by passing  to the limit. This is indeed possible thanks to Lemma \ref{lem:bound}. However, we need to face two difficulties: vanishing of solutions (that is, the limit solution is trivial) and concentration near the boundary (that is, the limit solution is defined in a half-space). The purpose of this section is to exclude both scenarios. 

Next result deals with the question of vanishing and is actually a version of Proposition 2.4 of \cite{bgs-cmp} adapted to problem \eqref{eq:GPstrip}.
\begin{prop}\label{lem:lemmacrucslab}
	Let $\psi$ be a finite energy solution of \eqref{eq:GPstrip} with $0<c<\sqrt{2}$, then
	$$\|1-|\psi|\|_{L^{\infty}(\Omega_N)}\geq \frac{2}{5}(1-\frac{c}{\sqrt{2}}).$$
\end{prop}

The proof is actually the same as in \cite{bgs-cmp}, with one difference: when integrating by parts, the authors use the decay estimates of the solutions to avoid contributions from infinity, and those estimates are available only for the Euclidean space case. Instead, here we use integrability bounds.

In our argument we will use liftings of the solutions, that is, we write  $\psi = \rho e^{i \theta}$. The existence of liftings is always guaranteed, for instance, if $|\psi(x)|\neq 0$ for all $x$.

\subsection{Liftings for solutions in $\Omega_{N}$ without vortices}

We consider here solutions without vortices, i.e. that do not vanish. The energy density  is given by the following formula
\[
e(\rho, \theta)=\frac 12  \left(|\nabla \rho|^2 + |\nabla \theta|^2\rho^2  \right)+\frac 14  \left(1-|\rho|^2\right)^2 
\]
and the associated energy is 
\[
\mathcal{E}(\rho, \theta):=\int_{\Omega_N} e(\rho, \theta) 
\]
By using the fact that $\psi=\rho e^{i \theta}$ is a solution of \eqref{eq:GPellip}, $\rho, \theta$ fulfill the following system of equations
\begin{equation}\label{eq:GPstripvarrp}
\left\{ \begin{aligned}
&\frac{c}{2}\partial_{x_1} \rho^2+\nabla \cdot(\rho^2\nabla \theta)=0,\\
& c \rho \partial_{x_1}\theta-\Delta \rho- \rho(1-\rho^2)+\rho|\nabla \theta|^2=0. 
\end{aligned}\right..
\end{equation}
The following pointwise inequality (Lemma 2.3 in BGS)
\begin{equation}\label{eq:pointcruc}
\left | (\rho^2-1)\partial_{x_1}\theta \right|\leq\frac{\sqrt{2}}{\rho} e(\rho, \theta)
\end{equation}
that holds for arbitary $C^1$ scalar function that can be written as  $\psi=\rho e^{i \theta}$ (not necessary being a solution) are crucial in the sequel.

\begin{lem}\label{prop:rhotheta}
Let $\psi$ be a vortexless finite energy solution in $\Omega_N$, then
$1-\rho$ and $\theta$ belong to $H^1_0(\Omega_N)$.
\end{lem}
\begin{proof}
Let us notice that for a vortexless finite energy solution  

\begin{equation} \label{energy-lifting} \mathcal{E}(\psi)=\int_{\Omega_N }|\nabla \rho|^2 + \rho^2 |\nabla \theta|^2 < +\infty \end{equation}
which implies, by means of Poincar\'{e} inequality,  that $\rho-1\in H^1_0(\Omega_N)$. Since $\nabla \rho$ bounded in $L^{\infty}$ (by Lemma \ref{lem:bound}), one concludes that $\rho(x) \to 1$ uniformly as $|x| \to +\infty$. Hence $\rho(x) \geq \rho_0 >0$ for all $x \in \Omega_N$. Again by \eqref{energy-lifting}, $\nabla \theta \in L^2(\Omega_N)$. To conclude the proof we shall prove that $\theta=0$ on $\partial \Omega_N$ which will allows to use Poincar\'{e} inequality.

\medskip 
Let us assume that $\theta=0$ if $x_1=-N$ and that $\theta=2 \pi$  if $x_1=N$. For any $\tilde{x} \in \R^{d-1}$ there exists $y \in (-N,N)$ such that $u(y, \tilde x)=0$.
We get
$$1=| u(y, \tilde x)-u(-N, \tilde x)|^2=\left| \int_{-N}^{y} \partial_{x_1} u(s,y) ds \right|^2\leq 2N  \int_{-N}^{N} \left| \partial_{x_1} u(s,y)\right|^2 ds.$$
By Fubini we get
$$ \int_{\Omega_N} \left| \partial_{x_1} u\right|^2  =\int_{\R^{d-1}}\left(  \int_{-N}^{N} \left| \partial_{x_1} u(s,y)\right|^2 ds\right) dy = + \infty,$$
which implies that the energy is infinity.

\end{proof}

From \eqref{eq:GPstripvarrp} we derive three useful identities that are important in the sequel. These identities have been stablished in \cite[Lemmas 2.8, 2.10]{bgs-cmp} for solutions in the whole euclidean space: here we adapt these arguments to the problem in the domain $\Omega_N$.

\begin{lem}\label{prop:stimmoment}
Let $\psi$ be a vortexless finite energy solution of \eqref{eq:GPstrip}. Then:

\begin{equation}\label{eq:MOwv}
\mathcal{P}(\psi)=\frac 12 \int_{\Omega_N}(1-\rho^2)\partial_{x_1}\theta.
\end{equation}

\begin{equation}\label{eq:MOwv2}
c \mathcal{P}= \int_{\Omega_N}\rho^2|\nabla \theta|^2.
\end{equation}
\begin{equation}\label{eq:MOwv3}
\int_{\Omega_N} \left(2 \rho|\nabla \rho|^2+\rho(1-\rho^2)^2\right)= c \int_{\Omega_N} \rho(1-\rho^2)\partial_{x_1}\theta+ \int_{\Omega_N} \rho(1-\rho^2)| \nabla \theta|^2.
\end{equation}
\end{lem}
\begin{proof}
	
Straightforward computation gives 
$$\mathcal{P}(\psi)=\frac 12 \int_{\Omega_N}\partial_{x_1}(\rho \sin \theta)-\rho^2\partial_{x_1}\theta=\frac 12 \int_{\Omega_N}\partial_{x_1}(\rho \sin \theta -\theta)+(1-\rho^2)\partial_{x_1}\theta.$$
We will prove that $\int_{\Omega_N}\partial_{x_1}(\rho \sin \theta -\theta)=0$. Thanks to Lemma \ref{prop:rhotheta} $\psi_1=(1-\rho^2)\partial_{x_1}\theta$ is integrable in $\Omega_N$ and hence we derive that $\psi_2=\partial_{x_1}(\rho \sin \theta -\theta)$ is integrable as well. By integration by parts together with Lemma \ref{prop:rhotheta} we get
$$\int_{\Omega_{N}} \psi_2 =\int_{\partial \Omega_{N} }\left(\rho \sin \theta -\theta\right)\eta_1 =0.$$

To get \eqref{eq:MOwv2} 
we multiply the first equation of \eqref{eq:GPstripvarrp} by $\theta$ and we integrate in $\Omega_{N,M}$, defined as

$$ \Omega_{N,M}=\left\{x\in \R^d, \ \ \ -N<x_1<N, \ |x_j|<M, \ \ 2 \leq j \leq d \right\} \subset \Omega_N.$$

By integrating by parts we obtain:

$$\frac{c}{2}\int_{\Omega_{N,M}} (1-\rho^2)\partial_{x_1}\theta-\int_{\Omega_{N,M}} \rho^2|\nabla \theta|^2  $$$$ =\int_{\partial \Omega_{N,M}} \theta \left ( \frac{c}{2} (1-\rho^2) \eta_1-  \rho^2 \nabla \theta \cdot \eta \right ).$$

Observe that by Lemma \ref{prop:rhotheta} all functions involved in the expression above belong to $L^1(\Omega_N)$, and recall that $\theta=0$ on $\partial \Omega_N$. Then, there exists a sequence $M_n$ such that 
$$\lim_{n \rightarrow \infty} \int_{\partial \Omega_{N,M_n}} \theta   \rho^2 \nabla \theta \cdot \eta=0.$$
This proves \eqref{eq:MOwv2}.
%If we now integrate  \eqref{eq:GPstripvarrp}  on $\Omega_{N,M}$ and integrate by parts we obtain that
%
%$$ 0=\int_{\partial \Omega_{N,M}}  \frac{c}{2} (1-\rho^2) \eta_1-  \rho^2 \nabla \theta \cdot \eta dx =\int_{\partial \Omega_{N,M}}   \rho^2 \nabla \theta \cdot \eta dx .$$

%using the fact the energy is finite,  we have that $\theta \frac{c}{2} (1-\rho^2)$  belongs to $W^{1,1}(\Omega_N)$.
%Indeed $||\theta (1-\rho^2)||_{L^1(\Omega_N)}\leq ||\theta ||_{L^2(\Omega_N)}||(1-\rho^2) ||_{L^2(\Omega_N)}<+\infty$ where we use the elementary inequality
%$(1-\rho^2)=(1-\rho)(1+\rho)$.  The term $\theta \rho^2 \nabla \theta$  belongs to $L^1(\Omega_N)$ as well. Therefore
%$$\lim_{n \rightarrow \infty }\int_{\partial \Omega_{N,M_n}} \theta \left ( \frac{c}{2} (1-\rho^2) \eta_1-  \rho^2 \nabla \theta \cdot \eta \right )dx=0.$$
By multiplying the second equation of \eqref{eq:GPstripvarrp} by $\rho^2-1$ and integrating over $\Omega_{N,M}$ by parts we obtain
$$\int_{\Omega_{N,M}} \left(2 \rho|\nabla \rho|^2+\rho(1-\rho^2)^2\right)+\int_{\partial \Omega_{N,M}} (1-\rho^2)\nabla \rho \cdot \eta = $$
$$=c \int_{\Omega_{N,M}} \rho(1-\rho^2)\partial_{x_1}\theta + \int_{\Omega_{N,M}} \rho(1-\rho^2)| \nabla \theta|^2.$$

Again by Lemma \ref{prop:rhotheta}, all functions involved in the above expression belong to $L^1(\Omega_N)$. Hence we can finde a sequence $M_n$ such that 
$$\lim_{n \rightarrow \infty} \int_{\partial \Omega_{N,M_n}}(1-\rho^2)\nabla \rho \cdot \eta=0.$$

This proves \eqref{eq:MOwv3} passing to the limit.
\end{proof}

\subsection{Proof of Proposition \ref{lem:lemmacrucslab}}
Let us call $\delta=||1-|\psi|||_{L^{\infty}(\Omega_N)}$. If $\delta>\frac 12>\frac{2}{5}(1-\frac{c}{\sqrt{2}})$ there is nothing to prove.  Let us suppose hence  that $\delta<\frac 12$ which implies that
  $\rho(x) \geq 1 - \delta >\frac 12 $ for any $x \in \Omega_N$. In particular $\psi$ admits a lifting $\psi = \rho e^{i \theta}$.
We notice that
$$4(1-\delta)\left(\int_{\Omega_N }\frac 12  |\nabla \rho|^2 +\frac 14  \left(1-|\rho|^2\right)^2\right)\leq\int_{\Omega_N } 2 \rho |\nabla \rho|^2 +  \rho \left(1-|\rho|^2\right)^2$$
and thanks to \eqref{eq:MOwv3} we get
\begin{equation}\label{eq:stimenerg}
\int_{\Omega_N}e(\rho, \theta)\leq \frac{1}{4(1-\delta)} \int_{\Omega_N} \rho(1-\rho^2) \left( c \partial_{x_1}\theta + | \nabla \theta|^2 \right)+\frac 12 \int_{\Omega_N}\rho^2|\nabla\theta|^2
\end{equation}
The strategy is to estimate r.h.s of \eqref{eq:stimenerg} using the pointwise bound given by \eqref{eq:pointcruc}.
We have, thanks to \eqref{eq:MOwv2} and  \eqref{eq:pointcruc}
$$ \frac{c}{4(1-\delta)} \int_{\Omega_N} \rho(1-\rho^2)\partial_{x_1}\theta+\frac 12 \int_{\Omega_N}\rho^2|\nabla \theta|^2\leq\left(\frac{\sqrt{2}c}{4(1-\delta)}+\frac{\sqrt {2}c}{4} \right)\int_{\Omega_N} e(\rho, \theta)$$
and hence
$$ \frac{c}{4(1-\delta)} \int_{\Omega_N} \rho(1-\rho^2)\partial_{x_1}\theta +\frac 12 \int_{\Omega_N}\rho^2|\nabla \theta|^2\leq \frac{c}{\sqrt{2}(1-\delta)}\int_{\Omega_N} e(\rho, \theta).$$
Now we claim that
\begin{equation}\label{eq:stimaimazz}
\left|\int_{\Omega_N}\rho(1-\rho^2)|\nabla \theta|^2 \right|\leq 6 \delta\int_{\Omega_N} e(\rho, \theta) 
\end{equation}
such that we obtain
$$\int_{\Omega_N} e(\rho, \theta) \leq (\frac{c}{\sqrt{2}(1-\delta)}+\frac{3\delta}{2(1-\delta)})\int_{\Omega_N} e(\rho, \theta).$$
The fact that $e(\rho, \theta )\geq 0$ and that $1- (\frac{c}{\sqrt{2}(1-\delta)}+\frac{3\delta}{2(1-\delta)})\leq 0$ if $\delta \geq \frac{2}{5}(1-\frac{c}{\sqrt{2}})$ concludes the proof.
Now we prove claim \eqref{eq:stimaimazz}. Notice that
$$\left|\int_{\Omega_N}\rho(1-\rho^2)|\nabla \theta|^2 \right| \leq \delta \int_{\Omega_N}\rho(1+\rho)|\nabla \theta|^2.$$
Now, $\rho(1+\rho)\leq 3 \rho^2$ if $\rho\geq \frac 12$, such that thanks to \eqref{eq:MOwv2}
$$\left|\int_{\Omega_N}\rho(1-\rho^2)|\nabla \theta|^2 \right| \leq 3 \delta \int_{\Omega_N}\rho^2|\nabla \theta|^2 \leq \frac{3 \delta c}{2}\int_{\Omega_N} (1-\rho^2)\partial_{x_1}\theta \leq 3 \sqrt{2} \delta c\int_{\Omega_N} e(\rho, \theta).$$
The proof of the claim ends noticing that $0<c<\sqrt{2}$.

\subsection{Conclusion of the proof of Theorem \ref{teo:almost}}

Take $c \in E$, $c_0 \in (0,c)$ and $N_0$ given by Proposition \ref{min-max}.

By Proposition \ref{lem:lemmacrucslab}, there exists $\xi_N \in \Omega_{k(N)}$ such that $|\psi_{k(N)}(\xi_N)-1| \nrightarrow 0 $, where $\psi_{k(N)}$ are the solutions given by Proposition \ref{prop-crucial}. By the uniform bounds of Lemma \ref{lem:bound}, we can use Ascoli-Arzel\`{a} Theorem to obtain in the limit ($C^k$ locally) a nontrivial solution of the Gross-Pitaevskii equation $\psi_c$. By Fatou Lemma, $\mathcal{E}(\psi_c) $ is finite. Moreover, by Remark \ref{remark morse}, $ind(\psi_c) \leq 1$.

Finally, by Fatou lemma and \eqref{eq:A bounded}, 

\begin{align*} I(\psi_c)= \frac{1}{d-1}  \sum_{j=2}^d \int |\partial_{x_j} \psi_c |^2 \leq  \frac{1}{d-1} \liminf_{N \to +\infty}  \sum_{j=2}^d \int |\partial_{x_j} \psi_{k(N)} |^2 \\  = \liminf_{N \to +\infty}  I_{k(N)} (\psi_{k(N)} ) \leq  \liminf_{N \to +\infty} \gamma_{k(N)}(c) \leq \chi(c_0). \end{align*}

\medskip If $d (\xi_N, \partial \Omega_{k(N)}) $ is bounded, up to a subsequence, our limit solution $\psi_c$ is defined in a half-space $\{x \in \R^d: \ x_1 > - m\}$ or $\{x \in \R^d: \ x_1 <  m\}$, for some $m>0$, and $\psi_c=1$ on its boundary. Instead, if $d (\xi_N, \partial \Omega_{k(N)}) $ is unbounded, the solution $\psi_c$ is defined in the whole euclidean space $\R^d$. In next proposition we rule out the first possibility, and this concludes the proof of Theorem \ref{teo:almost}.

\begin{prop} \label{prop:half} Let $\psi$ be a finite energy solution of the problem:
	
\begin{equation}\label{eq:half}
\begin{array}{rcl}i c\partial_{x_1}\psi +\Delta \psi+\left(1-|\psi|^2\right)\psi & = & 0  \ \  \text{ on } \R^d_+, \\
\psi & = & 1 \  \text{ on } \R^d_+, \end{array}
\end{equation}
where $\R^d_+ = \{x \in \R^d:\ x_1>0\}.$ Then $\psi=1$.
\end{prop}
	
\begin{proof}
The proof follows well-known ideas that date back to \cite{el}. 
If $\psi$ is a finite energy solution, then $\nabla \psi$ and $(1-|\psi|^2)$ are functions in $L^2(\R^d_+)$. Since $\psi$ is in $L^\infty(\R^d_+)$ and is a strong solution, standard regularity results allow us to conclude that $D^2 \psi$ belongs to $L^2(\R^d_+)$. Hence we can multiply equation \eqref{eq:half} by $\partial_{x_1}\psi$ and integrate by parts, obtaining:

$$ c \int_{\R^d_+} \lan  (i \partial_{x_1}\psi) , \partial_{x_1}\psi \ran =0;$$

$$\int_{\R^d_+} \lan \Delta \psi,  \partial_{x_1}\psi \ran = \int_{\partial \R^d_+} \lan (\nabla \psi \cdot \nu ), \partial_{x_1} \psi \ran  - \int_{\R^d_+}  \frac{1}{2} \partial_{x_1} \left ( |\nabla \psi|^2 \right ) $$ $$= - \int_{\partial \R^d_+} | \partial_{x_1} \psi|^2  + \frac{1}{2}   \int_{\partial \R^d_+} |\partial_{x_1} \psi|^2= -\frac{1}{2}   \int_{\partial \R^d_+} |\partial_{x_1} \psi|^2; $$

$$ \int_{\R^d_+}  \left(1-|\psi|^2\right)\lan \psi, \partial_{x_1}\psi \ran = - \frac 1 4 \int_{\R^d_+}  \partial_{x_1} \left ( (1- |\psi|^2)^2 \right )=0. $$

These computations imply that:

$$\int_{\partial \R^d_+} |\partial_{x_1} \psi|^2 =0.$$

In other words, $\partial_{x_1} \psi=0$ in $\partial \R^d_+$. By unique continuation, we conclude that $\psi=1$.

\end{proof}

\begin{remark} \label{remark halfspace}
	
The proof of Proposition \ref{prop:half} breaks down if we consider the half-space $\{ x \in \R^d:\ x_j >0 \}$, if $j>1$. The reason is that we do not know if:

$$ \int_{\R^d_+} \lan (i \partial_{x_1}\psi) , \partial_{x_j}\psi \ran$$
may cancel. This is one of the reasons why we choose slabs as approximating domains, instead of expanding balls, for instance (a choice that would have had advantages from the point of view of compactness). The second reason is that in $\Omega_N$ the Pohozaev-type identity given in Lemma \ref{lem:poho2} does not involve boundary terms.
	
\end{remark}

\section{Proof of Theorem \ref{teo:3}}

In this section we prove the compactness criterion given in \ref{teo:3}. We start by the following result, which is independent of the dimension:

\begin{prop} 	\label{ascoli}
Let $d=2$ or $3$, $c_n \to c$, $c_n \in E$ where the set $E$ is given by Theorem \ref{teo:almost}. Let $\psi_n$ be the sequence of solutions provided by that theorem. Then there exists $\xi_n \in \R^d$ such that $\psi_n(\cdot - \xi_n)$ converges locally  in $C^k$ (up to a subsequence) to a nontrivial solution $\psi_0$ of \eqref{eq:GPellip}.	

\end{prop}

\begin{proof} Let $\psi$ be a finite energy solution of \eqref{eq:GPellip} with $0<c<\sqrt{2}$. Then there exists $\e= \e(c) >0$ such that
	\begin{equation} \label{puff} \|1-|\psi|\|_{L^{\infty}(\Omega_N)}\geq \e. \end{equation}

Statement \eqref{puff} is just  Proposition 2.4 of \cite{bgs-cmp}. Compare it with Proposition \ref{lem:lemmacrucslab}, which is nothing but its version for problem \eqref{eq:GPstrip} (with a slight change of the constants).

Then, there exists $\xi_n$ such that $ | 1- |\psi_n(\xi_n)||> \e$ for some fixed $\e>0$. By Lemma \ref{lem:bound} we can use Ascoli-Arzel\`{a} Theorem to obtain that $\psi_n(\cdot - \xi_n)$ converges locally in $C^k$ to a nontrivial solution $\psi_0$ of \eqref{eq:GPellip}.	
	
\end{proof}

The main problem to conclude the proof of Theorem \ref{teo:3} or \ref{teo:all} is to assure that $\psi_0$ has finite energy. Let us point out that the boundedness of the energy cannot be deduced only by using the Pohozaev identities given in Lemmas \ref{lem:poho}, \ref{lem:poho2}.

Observe that since $I_{c_n}(\psi_n )$ is bounded, Lemma \ref{lem:poho2} implies that
$$ \sum_{j=2}^3 \int_{\R^3}  |\partial_{x_j} \psi_n|^2 = O(1).$$

The idea of the proof is to try to relate the behavior of $\psi_n$ with that of the 1-D solutions of the Gross-Pitaevskii equation. Next proposition is a first step in this line (see also Remark \ref{rrr}).

\begin{prop} Let $\psi_n$ be solutions of \eqref{eq:GPellip} for $c_n$, $c_n \to c$, such that $I_{c_n}(\psi_n) \leq C$. Then,

$$ \int_{\R^3} |\nabla g_n|^2 +  \int_{\R^3} |\nabla h_n|^2 =O(1),$$

where 

\begin{equation} \label{g} g_n= (\partial_{x_1} u_n) v_n - (\partial_{x_1} v_n) u_n - \frac{c_n}{2} (\rho_n^2-1), \end{equation}

\begin{equation} \label{h} h_n= \frac 1 2 |\partial_{x_1} \psi_n|^2 - \frac 1 4  (1-\rho_n^2)^2. \end{equation}
\end{prop}

\begin{remark} \label{rrr}
	
The two quantities defined above correspond to the invariants of the 1-D Gross-Pitaevskii equation. Indeed $h$ represents its hamiltonian, whereas $g$ is another invariant given by the fact that the problem, after a change of variables, is radially symmetric (see equations \eqref{eq:GP1D}, \eqref{eq:GP1Dbis}). On this aspect, see for instance \cite{bgs-survey}, pages 3-4.
	
\end{remark}

\begin{proof}

For the sake of clarity we drop the subscript $n$ in the proof of this proposition.

We first consider the function $g$, which is an $L^2$ function, but with $L^2$ norm out of control. Observe that equation \eqref{eq:GPellip} implies that $\nabla \cdot G=0$, where

\begin{equation} \label{defG} G=(g, u_{x_2} v - v_{x_2} u, u_{x_3} v - v_{x_3} u). \end{equation}
	
Straightforward computations give:

$$ curl \, G =\left(  \begin{array}{c} \displaystyle 2 u_{x_3} v_{x_2} - 2 v_{x_3} u_{x_2} \\ \displaystyle 2 v_{x_3} u_{x_1} - 2 u_{x_3} v_{x_1} - \frac{c}{2} (\rho^2-1)_{x_3} \\ \displaystyle 
\displaystyle - 2 v_{x_2} u_{x_1} + 2 u_{x_2} v_{x_1} + \frac{c}{2} (\rho^2-1)_{x_2}  \end{array} \right).$$

Observe that by \eqref{eq:A bounded}, the derivatives with respect to $x_2$, $x_3$ are uniformly bounded (with respect to $n$) in $L^2$. Moreover, all factors involved are bounded in $L^\infty$ by Lemma \ref{lem:bound}. As a consequence, $curl G$ is uniformly bounded in $L^2$. Observe now that:

$$ G= curl (-\Delta^{-1} curl \, G ),$$

where $\Delta^{-1}$ is given by convolution with the Coulomb potential $\frac{1}{4 \pi |x|}$, see for instance \cite[Subsection 2.4.1]{bertozzi}. 

 By using the Fourier Transform and Plancherel, all partial derivatives of $G$ are uniformly bounded in $L^2$, independently of $n$. This concludes the proof for $g$.

\medskip 

For $h$, the proof follows the same ideas. Let us define the vector field:

$$H =\left( h, u_{x_1} u_{x_2} + v_{x_1} v_{x_2},  u_{x_1} u_{x_3} + v_{x_1} v_{x_3} \right).$$

Let us recall here that $|\psi_{x_1}|^2 = u_{x_1}^2 + v_{x_1}^2$. Observe first that $H$ is an $L^2$ vector field, even if its $L^2$ norm could be unbounded as $n \to +\infty$. Taking into account \eqref{eq:GPellip}, straightforward computations give:

$$ \nabla \cdot H = u_{x_1 x_2} u_{x_2} + u_{x_1 x_3} u_{x_3} + v_{x_1 x_2} v_{x_2} + v_{x_1 x_3} v_{x_3}   $$

which is uniformly bounded in $L^2$ norm, again, by \eqref{eq:A bounded} and Lemma \ref{lem:bound}. Moreover, we can compute:

$$ curl \, H = \left(  \begin{array}{c} \displaystyle u_{x_1 x_2} u_{x_3} + v_{x_1 x_2} v_{x_3} - u_{x_1 x_3} u_{x_2} -v_{x_1 x_3} u_{x_2} \\ \displaystyle - u_{x_1 x_1} u_{x_3} -  v_{x_1 x_1} v_{x_3} - (1-\rho^2) (u u_{x_3} + v v_{x_3})\\ \displaystyle   u_{x_1 x_1} u_{x_2} +  v_{x_1 x_1} v_{x_2} + (1-\rho^2) (u u_{x_2} + v v_{x_2}) \end{array} \right).$$

which is also uniformly bounded in $L^2$ norm. We now recall that:

$$ H = \nabla (\Delta^{-1} (\nabla \cdot H)) - curl (\Delta^{-1} \, curl H),$$

see again \cite[Subsection 2.4.1]{bertozzi}. By using the Fourier Transform and Plancherel, all partial derivatives of  $H$ are uniformly bounded in $L^2$, finishing the proof. 

\end{proof}

\begin{remark} \label{r1} Let us point out that the above result can be easily extended to any dimension. However, in dimension 3 it implies, by Sobolev inequality, that:

\begin{equation} \label{eq:Sob} \int_{\R^3} |g_n|^6 + \int_{\R^3} |h_n|^6 = O(1) \end{equation}

In dimension $d>3$ the Sobolev exponent is $\frac{2d}{d-2}$. However we cannot deduce a similar expression in dimension 2. The lack of a Sobolev inequality in dimension 2 is one of the obstacles for this approach to work also in the planar case.

\end{remark}

\begin{definition} \label{defS}  We define the set $S_n^r= \{ x \in \R^3: \rho_n(x) < r \}$. The behavior of these sets will be important in our arguments.
\end{definition}

Next lemma is a key ingredient in our proof.

\begin{lem} \label{lem:51} Under the assumptions of Theorem \ref{teo:3}, assume that for some $r\in (0, 1)$, $|S_n^{r}| \to +\infty$. Then, there exists $\xi_n \in S_n^r$ and $R_n \to +\infty$ such that:
	
	$$\int_{B(\xi_n, R_n)}  |g_n |^6 +  |h_n |^6  + \sum_{i=2}^3 |\partial_{x_i} \psi_n |^2  \to 0.$$
	
\end{lem}

\begin{proof} 
	Take $x^n_1 \in S_n^r$, and define $R_n = |S_n^r|^{\frac{1}{6}}$. Observe that
	
	$$|B(x^n_1, 2 R_n)| = c_0  |S_n^r|^{\frac{1}{2}}, \ \ c_0 = \frac{32}{3} \pi.$$
	
	As a consequence, there exists $x^n_2 \in S_n^r \setminus B(x^n_1, 2 R_n)$. Clearly, 
	
	$$ B(x^n_1, R_n) \cap B(x^n_2, R_n) = \emptyset \mbox { and } |B(x^n_1, 2 R_n) \cup B(x^n_2, 2 R_n) | \leq 2 c_0  |S_n^r|^{\frac{1}{2}}.$$
	
	We can choose then $x^n_3 \in S_n^r \setminus (B(x^n_1, 2 R_n) \cup B(x^n_2, 2 R_n))$, with
	
	$$ B(x^n_1, R_n) \cap B(x^n_2, R_n) \cap  B(x^n_3, R_n)= \emptyset $$
	and 
	$$ |B(x^n_1, 2 R_n) \cup B(x^n_2, 2 R_n) \cup  B(x^n_3, R_n)| \leq 3 c_0  |S_n^r|^{\frac{1}{2}}.$$
	
	\medskip In this way we find $x_n^1 \dots x_n^{j_n} \in S_n^r$ with:
	
	$$ B(x^n_j, R_n) \cap B(x^n_k, R_n) = \emptyset, \ \ j,\ k \in \{1, \dots j_n \}, \ j \neq k.$$
	
	where $j_n = [ \frac{|S_n|^{1/2}}{c_0}]$ (here $[a]$ denotes the largest integer smaller or equal than $a$). Hence we can choose $\xi_n = x_n^k$ such that, taking into account \eqref{eq:Sob} and \eqref{eq:A bounded}:
	
	$$ \int_{B(\xi_n, R_n)}  | g_n |^6 + | h_n |^6 + \sum_{i=2}^d |\partial_{x_i} \psi_n |^2  \leq  \frac{C}{j_n} \to 0.$$
	
\end{proof}
	
The above result will be the key to prove next proposition, which allows us to have some control on the set of vortices of the solutions.
	
\begin{prop} \label{prop:vortices} Let us fix $r \in (0, c/\sqrt{2})$. Then, there exists $N \in \mathbb{N}$ and $N$ sequences of disjoint closed balls  $\overline{B_k^n}= \overline{B}(\xi_k^n, R_k)$ ($k=1 \dots n$) with $R_k \in (1, N)$ such that 
	
$$ S_n^r \subset \cup_{k=1}^N B_k^n.$$

\end{prop}

\begin{proof} The proof is divided into several steps:
	
	\medskip
	
{\bf Step 1: }  $|S_n^r|$ remains bounded.
	
Assume by contradiction that $|S_n^r| \to +\infty$ as $n \to +\infty$. Take $\xi_n \in \R^3$ given by Lemma \ref{lem:51}, and define $\tilde{\psi}_n = \psi_n(\cdot - \xi_n)$. By Lemma \ref{lem:bound} we can use Ascoli-Arzel\`{a} theorem to conclude that, up to a subsequence, $\tilde{\psi}_n$ converges $C^k$ locally to a solution $\psi$ of \eqref{eq:GPellip}. By the choice of $\xi_n$, this solution satisfies that $\rho(0) \leq r$. Moreover, by Fatou lemma we have that:
	
$$ \partial_{x_2} \psi =0, \ \partial_{x_3} \psi =0, \ g=0, \ h =0,$$
where $g$ and $h$ are the analogous of \eqref{g}, \eqref{h}, namely:

\[ g= u_{x_1} v - v_{x_1} u - \frac{c}{2} (\rho^2-1), \]

\[h= \frac 1 2 |\psi_{x_1}|^2 - \frac 1 4  (1-\rho^2)^2. \]

As a consequence, $\psi$ is a 1-D solution to the Gross-Pitaevskii equation with $g=0$, $h=0$. But those are precisely the finite energy 1-D travelling waves (see \cite[pages 3, 4]{bgs-survey}); hence, after a rotation, $\psi$ has the explicit expression:

$$\psi(x_1)=\sqrt{\frac{2-c^2}{2}} \tanh \left( \frac{\sqrt{2-c^2}  }{2}(x_1 + t) \right)+i\frac{c}{\sqrt{2}},  \ t \in \R.$$

But this is in contradiction with $|\psi(0)| \leq r < c/\sqrt{2}$, concluding the proof.

\medskip

{\bf Step 2:} There exists $N \in \mathbb{N}$, $\xi_k^n \in \R^3$ such that:

 $$ S_n^r \subset \cup_{k=1}^N B(\xi_k^n, 1).$$

Fix $s \in (r, \frac{c}{\sqrt{2}})$, and define $\xi_1^n$ as any point in $S_n^r$. Since $\nabla \rho_n$ is uniformly bounded (Lemma \ref{lem:bound}), there exists $\delta >0$ such that $B(\xi_1^n, \delta) \subset S_n^s$. It suffices to take
$$ \delta \leq  \frac{s-r}{sup_{n} \| \nabla \rho_n \|_{L^{\infty}}}.$$
Without loss of generality we can assume that $\delta < 1/2$.

Take now $\xi_2^n$ any point in $S_n^r \setminus B(\xi_1^n, 1)$; again, $B(\xi_2^n, \delta) \subset S_n^s$, and observe that $B(\xi_1^n, \delta) \cap B(\xi_2^n, \delta) = \emptyset$.

We follow by taking $\xi_3^n$ any point in $S_n^r \setminus \left( B(\xi_1^n, 1) \cup B(\xi_2^n, 1) \right)$, if there exists one. Since $|S_n^s|$ is bounded by the step 1, this procedure has to finish at a certain point, yielding the thesis of the proposition. Indeed we cannot find more than $N$ such points, where 

$$ N = \left [ \frac{\sup_n |S_n^s|}{\frac 4 3 \pi \delta^3} \right ].$$

Recall that $[a]$ stands for the largest integer smaller or equal than $a$.

\medskip

{\bf Step 3:} Conclusion

By step 2, we have already $S_n^r$ contained in $N$ balls of radius 1. The problem is that they might not be disjoint. We now make a procedure of aggregation of balls which is gererally described as follows:

\medskip Take a closed ball $\overline{B}(x, R_x)$. If it intersects a closed ball $\overline{B}(y,R_y)$, we replace both balls by $\overline{B}(x, R_x+R_y)$. We now repeat the procedure to the new set of balls.

\medskip We apply this procedure iteratively to the balls given in Step 2, and in this way we conclude.

\end{proof}

The above proposition is the first milestone in our proof: it allows us to control the vortices of the solutions, as they are always contained in a fixed number of disjoint balls of bounded radii. Since $\R^3 \setminus \cup_{k=1}^N \overline{B}_k^n$ is a simply connected open set, we can guarantee the existence of a lifting of $\psi_n$ outside these balls. Being more specific, taking $\frac{c}{2}$ as the value $r$ (for instance), we can write: 

$$\psi_n(x) = \rho_n(x) e^{i \theta_n(x)} \ \forall  \ x \in \R^d \setminus \cup_{k=1}^N B_k^n$$
where the balls $B_k^n$ are given by Proposition \ref{prop:vortices}. Since $\psi_n$ is a solution of \eqref{eq:GPellip}, we have that $\rho_n$, $\theta_n$ satisfy equations \eqref{eq:GPstripvarrp}.

\begin{remark} \label{r2} This is a second crucial point in which the requirement $d\geq 3$ is crucial. If $d=2$ we can have liftings of finite energy solutions outside one ball (see \cite[Lemma 15]{gravejat-AIHP}), but this is not possible in the complement of two or more disjoint balls.
\end{remark}

Next lemma is inspired in \cite{bgs-cmp}[Lemmas 2.8, 2.10], which are concerned with the case without vortices.   Compare it with the identities \eqref{eq:MOwv}, \eqref{eq:MOwv2} for the vortexless case.

\begin{lem} \label{lem:O(1)}
	Take $c \in (0,\ \sqrt{2})$, $c_n \in E$ with $c_n \to c$ and $\psi_n$ the solutions given by Theorem \ref{teo:almost}. Then
	\begin{equation}\label{eq:stimmomassa}
	\mathcal{P}(\psi_n)=\frac 12 \int_{\R^3 \setminus \cup_{k=1}^N B_k^n}(1-\rho_n^2)\partial_{x_1}\theta_n +O(1),
	\end{equation}
	\begin{equation}\label{eq:stimmomass2b}
	c \mathcal{P}(\psi_n)= \int_{\R^3 \setminus \cup_{k=1}^N B_k^n}\rho_n^2|\nabla \theta_n|^2 +O(1),
	\end{equation}
	\begin{equation}\label{eq:stimmomass3c}
	\int_{\R^3 \setminus \cup_{k=1}^N B_k^n}|\nabla \rho_n|^2 =O(1).
	\end{equation}
\end{lem}
\begin{proof}
	
    First of all, observe that $$\nabla \theta = \frac{u \nabla v-v \nabla u}{\rho^2},$$ so that
	
	\begin{equation} \label{boundtheta} |\nabla \theta_n | = O(1) \mbox{ in } \partial B_k^n  \Rightarrow |\theta_n(p) - \theta_n(q)| \leq C \ \forall \ p,\ q \in \partial B_{k}^n. \end{equation}

This is useful in what follows; observe that we do not know whether $\| \theta_n \|_{L^{\infty}}$ is bounded or not.
	
Direct computation gives
	$$
	\mathcal{P}(\psi_n)=\frac 12 \int_{\R^3 \setminus \cup_{k=1}^N B_k^n}\partial_{x_1}(\rho_n \sin \theta_n)-\rho_n^2\partial_{x_1}\theta +\frac 12 \int_{\cup_{k=1}^N B_k^n} \lan i \partial_{x_1} \psi_n, \psi_n-1\ran 
	$$
	which implies that 
	\begin{equation}\label{eq:stimmomassz}\mathcal{P}(\psi_n)=\frac 12 \int_{\R^3 \setminus \cup_{k=1}^N B_k^n}\partial_{x_1}(\rho_n \sin \theta_n -\theta_n)+(1-\rho_n^2)\partial_{x_1}\theta_n +O(1).
	\end{equation}
	In order to get \eqref{eq:stimmomassa} it suffices hence to prove that
	\begin{equation}\label{eq:stimmomasszz} \int_{\R^3 \setminus \cup_{k=1}^N B_k^n}\partial_{x_1}(\rho_n \sin \theta_n-\theta_n)=O(1).
	\end{equation}

	We recall that $ \partial_{x_1} (\rho_n \sin \theta_n - \theta_n)$ is integrable thanks to \eqref{eq:pointcruc} and \eqref{eq:stimmomassz}. By integration by parts, using the the decay estimates at the infinity,  we get
	$$\int_{\R^3 \setminus \cup_{k=1}^N B_k^n} \partial_{x_1}(\rho_n \sin \theta_n-\theta_n) =\sum_{k=1}^N \int_{\partial B_k^n} \left(\rho_n \sin \theta_n -\theta_n\right)\eta_1,$$
	where $\eta_1$ is the first component of the inward unit normal vector to the spheres $B_k^n$. Relation  \eqref{eq:stimmomasszz} follows now from \eqref{boundtheta} together with the fact that the outward unit surface normal $\eta_1$ has zero average on the sphere, which implies that
	$$\int_{\partial B_k^n} \theta_n\eta_1=\int_{\partial B_k^n} \left(\theta_n-\theta_n(p_0)\right)\eta_1=O(1),$$
	where $p_0$ is an arbitrary point on the sphere $B_k^n$.\\
	
	In order to prove \eqref{eq:stimmomass2b} we argue as in Lemma \ref{prop:stimmoment}, i.e. multiplying per first equation of \eqref{eq:GPstripvarrp} by $\theta_n$ and then
	integrating on $R^3 \setminus \cup_{k=1}^N B_k^n$. By integration by parts we get
	$$\frac{c}{2}\int_{R^3 \setminus \cup_{k=1}^N B_k^n} (1-\rho_n^2)\partial_{x_1}\theta_n -\int_{R^3 \setminus \cup_{k=1}^N B_k^n} \rho_n^2|\nabla \theta_n|^2  $$$$ = \sum_{k=1}^N \int_{\partial B_k^n} \theta_n \left ( \frac{c}{2} (1-\rho_n^2) \eta_1-  \rho_n^2 \nabla \theta \cdot \eta \right ).$$
	
	Observe that $G_n(x)= (\frac{c}{2} (1-\rho_n^2), 0, 0 ) - \rho_n^2 \nabla \theta_n$, where $G_n$ is defined in \eqref{defG}. In particular it is defined in the whole euclidean space and $\nabla \cdot G_n=0$. By integrating by parts in $B_k^n$, we obtain that
	
	$$ \int_{\partial B_k^n}  \frac{c}{2} (1-\rho_n^2) \eta_1-  \rho_n^2 \nabla \theta \cdot \eta =0.$$
	
	As a consequence, we can use \eqref{boundtheta} to obtain:  
	
	$$ \int_{\partial B_k^n} \theta_n \left ( \frac{c}{2} (1-\rho_n^2) \eta_1-  \rho_n^2 \nabla \theta \cdot \eta \right ) $$$$ = \int_{\partial B_k^n} (\theta_n - \theta_n(p_0)) \left ( \frac{c}{2} (1-\rho_n^2) \eta_1-  \rho_n^2 \nabla \theta \cdot \eta \right )=O(1).$$
	
	\medskip 
	
	Now we prove \eqref{eq:stimmomass3c}. From Lemma \ref{lem:poho} we get
	
	$$\frac{1}{2} \int_{\R^3} |\nabla \psi_n|^2  -2 c \mathcal{P}(\psi_n)+ \frac 3 4  \int_{\R^3} \left(1-|\psi_n|^2\right)^2 =0,$$
	which implies, thanks to \eqref{eq:stimmomass2b}
	$$\frac{1}{2} \int_{\R^3 \setminus \cup_{k=1}^N B_k^n} |\nabla \rho_n|^2  +\frac{3}{2} \int_{\R^3 \setminus \cup_{k=1}^N B_k^n} \rho^2|\nabla \theta_n|^2 + \frac 3 4  \int_{\R^3 \setminus \cup_{k=1}^N B_k^n} \left(1-|\rho_n|^2\right)^2 =$$
	$$=3  \int_{\R^3 \setminus \cup_{k=1}^N B_k^n} \rho^2|\nabla \theta_n|^2  +O(1)=3c  \mathcal{P}(\psi_n)+O(1).$$
	As a consequence we get
	$$3 \left(\mathcal{E}(\psi_n)-c  \mathcal{P}(\psi_n)\right)=\int_{\R^3 \setminus \cup_{k=1}^N B_k^n} |\nabla \rho_n|^2+O(1)$$
	and hence  \eqref{eq:stimmomass3} follows by the fact that $\mathcal{E}(\psi_n)-c  \mathcal{P}(\psi_n)=I(\psi_n)=O(1).$
\end{proof}

For next proposition it is useful to recall the definition \ref{defS}.

\begin{prop} \label{prop:ma}	Take $c \in (0,\ \sqrt{2})$, $c_n \in E$ with $c_n \to c$ and $\psi_n$ the solutions given by Theorem \ref{teo:almost}. Assume that:

\begin{equation}  \mathcal{E}(\psi_n) \to + \infty. \end{equation}
	
Then, for any $r \in (\frac{c}{\sqrt{2}}, 1)$, $|S_n^r| \to +\infty$.
	
\end{prop}

\begin{proof} We recall the form of the energy for functions $\psi$ given by a lifting $\psi = \rho e^{i \theta}$:

	$$e(\rho, \theta)= \frac 12  \left(|\nabla \rho|^2 + |\nabla \theta|^2\rho^2  \right)+\frac 14  \left(1-|\rho|^2\right)^2.$$
	
	The following function represents the lagrangian in the vortexless case, and is an approximation of the real lagrangian in view of \eqref{eq:stimmomassa}:
	
	$$l(\rho, \theta)= e(\rho, \theta)- \frac{c_n}{2} (1-\rho^2)\partial_{x_1} \theta.$$
	
	Assume by contradiction that $|S_n^r|$ is bounded for some $r > \frac{c}{\sqrt{2}}$. Observe that:
	
	$$ I^{c_n} (\psi_n ) = \int_{\R^3 \setminus \cup_{k=1}^N B_k^n} l(\rho_n, \theta_n) + O(1)= \int_{\{\rho_n \geq r\}} l(\rho_n, \theta_n) + O(1).$$
	
	We now use the inequality $|\frac c 2 (1-\rho_n^2)\partial_{x_1} \theta_n| \leq \frac{(1-\rho^2)^2}{4 (1+\e)} + \frac{c^2}{4} (\partial_{x_1} \theta_n)^2(1+\e)$ with suitable $\e>0$ to obtain:
	
	$$ \int_{\{\rho_n \geq r\}} l(\rho_n, \theta_n) \geq  \int_{\{\rho_n \geq r \} } \frac 1 2  |\nabla \rho_n|^2 + \left[ \frac 1 2 - \frac{c^2 (1+\e)}{4\rho_n^2} \right] |\nabla \theta_n|^2\rho_n^2 + \frac{\e}{1+\e}  \frac{(1-\rho_n^2)^2}{4}$$ $$ \geq  \e_0 \int_{\{\rho_n \geq r \}} e(\rho_n, \theta_n) = \e_0 \, \mathcal{E}(\psi_n) + O(1),$$
	
	for suitable  $\e_0>0$. Then,
	
	$$O(1) = I^{c_n}(\psi_n) \geq \e_0 \, \mathcal{E}(\psi_n) + O(1),$$
	and this allows us to conclude.

\end{proof}

\subsection{Proof of Theorem \ref{teo:3}}

With all the results above we can inmediately conclude the proof of Theorem \ref{teo:3}. Indeed, by \eqref{eq:stimmomass3} and Sobolev inequality, we have that

$$ \int_{\R^3} (1-\rho_n)^6 = O(1).$$

If $\mathcal{E}(\psi_n) \to +\infty$, Proposition \ref{prop:ma} implies $|S_n^r|$ is unbounded for $r > \frac{c}{\sqrt{2}}$, and this is a contradiction with the above estimate. Hence $\mathcal{E}(\psi_n)$ is bounded. By Fatou Lemma, the solution $\psi_0$ given in Proposition \ref{ascoli} has finite energy, concluding the proof.

\section{Proof of Theorem \ref{teo:all}}

In this section we prove the compactness criterion given in Theorem \ref{teo:all}. The proof follows some of the ideas of the previous section, but with important differences. As previously, we will be done if we show that $\mathcal{E}(\psi_n)$ is bounded. 

By \eqref{control}, we have that $\psi_n \neq 0$ outside $B(0,R)$; as a consequence, 
$\psi_n$ admit a lifting $\psi_n(x) = \rho_n(x) e^{i \theta_n(x)}$ for all  $x \in \R^2 \setminus B(0,R)$. This is a consequence of the fact that $\psi_n$ have finite energy, see \cite{gravejat-AIHP}[Lemma 15]. In the vortexless case, this lifting holds in the whole euclidean space.

Next lemma is a version of Lemma \ref{lem:O(1)}:

\begin{lem} \label{jopeta}
 Take $c \in (0,\ \sqrt{2})$, $c_n \in E$ with $c_n \to c$ and $\psi_n$ the solutions given by Theorem \ref{teo:almost}. Assume also that there exists $R>0$ and $\delta>0$ such that \eqref{control} is satisfied. Then
\begin{equation}\label{eq:stimmomass}
\mathcal{P}(\psi_n)=\frac 12 \int_{B(0,R)^c}(1-\rho_n^2)\partial_{x_1}\theta_n +O(1),
\end{equation}
\begin{equation}\label{eq:stimmomass2}
c \mathcal{P}(\psi_n)= \int_{B(0,R)^c}\rho_n^2|\nabla \theta_n|^2 +O(1),
\end{equation}
\begin{equation}\label{eq:stimmomass3}
 \int_{B(0,R)^c}|\nabla \rho_n|^2 =O(1).
\end{equation}
\end{lem}
\begin{proof}

The proof is completely analogue to that of Lemma \ref{lem:O(1)}. Observe that in the vortexless case we have exact identities in \eqref{eq:stimmomass}, \eqref{eq:stimmomass2}, \eqref{eq:stimmomass3}.

\end{proof}.

With Lemma \ref{lem:O(1)} in hand, we can adapt the proof of Proposition \ref{prop:ma} to our setting, obtaining the following result:

\begin{prop} \label{prop:ma2}	Take $c \in (0,\ \sqrt{2})$, $c_n \in E$ with $c_n \to c$ and $\psi_n$ the solutions given by Theorem \ref{teo:almost}. Assume that:
	
	\[ \mathcal{E}(\psi_n) \to + \infty. \]
	
	Then, for any $r \in (\frac{c}{\sqrt{2}}, 1)$, $|S_n^r| \to +\infty$.
	
\end{prop}

Next result is analogue to Lemma \ref{lem:51}. The only difference is that now we do not know that \eqref{eq:Sob} holds, but instead we have \eqref{eq:stimmomass3}.

\begin{lem} \label{lem:61} Under the assumptions of Theorem \ref{teo:all}, assume that for some $r\in (0, 1)$, $|S_n^{r}| \to +\infty$. Then, there exists $\xi_n \in S_n^r$ and $R_n \to +\infty$ such that:
	
	$$\int_{B(\xi_n, R_n)}  |\nabla \rho_n |^2 + |\partial_{x_2} \psi_n |^2  \to 0.$$
	
\end{lem}

\begin{proof}
	
	The proof is analogue to that of Lemma \ref{lem:51}.
	
\end{proof}

\subsection{Proof of Theorem \ref{teo:all}}

Assume by contradiction that $\mathcal{E}(\psi_n) \to +\infty$. By Proposition \ref{prop:ma2}, we can apply Lemma \ref{lem:61} to a value $r$ satisfying that:

$$\frac{c}{\sqrt{2}} < r < \sqrt{ \frac{2}{3} (1+c^2/4)} <1.$$ Notice that this is possible if $c<\sqrt{2}$. 

Let $\xi_n \in \R^d$ given by Lemma \ref{lem:61} and define $\tilde{\psi}_n(x)= \psi_n(x- \xi_n)$. Up to a subsequence we have that:

$$ \tilde{\psi}_n \to {\psi}_0 \mbox{ in } C^k_{loc}(\R^d).$$

Taking into account Remark \ref{remark morse}, $ind(\psi_0) \leq 1$. By Lemma \ref{lem:51}, $\psi_0$ depends only of the $x_1$ variable. Moreover $\nabla \rho_0 =0$ where $\rho_0 = |{\psi}_0| \leq r$. That is, $\psi_0 (x_1)$ is a $1D$ circular solution,

$$ \psi_0(x_1) = \rho_0 e^{i \omega  (x_1 - t)},$$
where  $\omega^2 + c \omega + \rho_0^2=1$. By the choice of $r$, we have that $\rho_0^2 < \frac{2}{3} (1+c^2/4)$. But those solutions have infinite Morse index, as shown in Proposition \ref{appendix} (see Appendix). This contradiction shows that $\mathcal{E}(\psi_n)$ is bounded.

 By Fatou Lemma, the solution $\psi_0$ given in Proposition \ref{ascoli} has finite energy, concluding the proof.

\begin{remark} Let us point out that Theorem \ref{teo:3} does not need the information on the Morse index of the solutions. The main tool there is that $I^{c_n}(\psi_n) = O(1)$. Instead, Theorem \ref{teo:all} requires in a essential way that the Morse index of the solutions obtained is bounded.
\end{remark}

\section{Appendix (by Rafael Ortega)}

In this appendix we prove the following result:

\begin{prop} \label{appendix} Given $t \in \R$, $\omega_0 \in \R$, $\rho_0 >0$ satisfying that $\omega_0^2 + c \omega_0 + \rho_0^2=1$, the function $\psi_0(x) = \rho_0 e^{i \omega_0  (x - t)}$ is a (infinite energy) solution of \eqref{eq:GPellip}. Assume also that $\rho_0^2 < \frac{2}{3} (1+c^2/4)$. Then its Morse index, as defined in Definition \ref{Morse}, is infinity. 
\end{prop}

\begin{proof} The problem is autonomous so that we can assume $t=0$. The proof is based on the study of the $1D$ problem:

\begin{equation}\label{eq:GP1D}
 \psi'' + i c \psi' + \left(1-|\psi|^2\right)\psi=0  \ \  \text{ on } \R.
\end{equation}

By the change of variables $\phi= e^{i x c/2 } \psi$ we pass to a problem:

\begin{equation}\label{eq:GP1Dbis}
\phi'' +\left(1+ c^2/4 -|\phi|^2\right)\phi=0  \ \  \text{ on } \R.
\end{equation}

The Morse index of this problem depends on the existence of conjugate points to some solutions of the linearized equation, see for instance \cite[Chapter 5]{gelfand}. The function $\phi(x)= \rho_0 e^{i \omega_1 x}$ is a solution of \eqref{eq:GP1Dbis}, where , $\omega_1 = \omega_0 + c/2$. Observe that

\begin{equation} \label{relation}
\omega_1^2 + \rho_0^2 = 1 + c^2/4
\end{equation} 

The linearized equation to \eqref{eq:GP1Dbis} around the solution $\phi$ is:

\[ \zeta'' + (1+c^2/4) \zeta - 2 \overline{\phi(s)} \phi(s) \zeta - \phi(s)^2 \overline{\zeta}=0.
\]

We will follow the lines of \cite[Section 21]{SM} to analyze the oscillatory properties of this equation.

\[ \zeta'' + (1+c^2/4) \zeta - 2 \rho_0^2 \zeta - \rho_0^2 e^{2 i \omega_1 s }\overline{\zeta}=0.
\]

We now make the change of variable $\zeta = e^{i \omega_1 s} \eta$, to obtain a constant coefficient linear system:

\begin{equation} \label{constant}
\eta'' + 2 i \omega_1 \eta' - \rho_0^2 \eta - \rho_0^2 \overline{\eta}=0.
\end{equation}

%Let us denote $\eta= \eta_1 + i \, \eta_2$, $\sigma_i= \eta_i'$, and $\alpha = (\eta_1,\, \eta_2,\, \sigma_1,\, \sigma_2)$. Then
%\eqref{constant} reduces to solve the equation $\alpha'(s)= A \cdot \alpha(s)$, where:
%
%$$A= \left( \begin{array}{cccc} 0 & 0 & 1 & 0 \\ 0 & 0 & 0 & 1 \\ 2 \rho_0^2 & 0 & 0 & 2 \omega_1 \\ 0 & 0 & -2 \omega_1 & 0 \end{array} \right).$$
%
%The eigenvalues of $A$ are $0$ (with algebraic multiplicity 2) and $\pm \sqrt{2 \rho_0^2 - 4 \omega_1^2} $. Hence one can expect the existence of conjugate points when 

If $\rho_0^2 < 2 \omega_1^2$ (which, by \eqref{relation}, reduces to $\rho_0^2 < \frac{2}{3}(1+c^2/4)$) we can find the explicit solution to \eqref{constant}:

$$ \eta(s)= \frac{ \sin \left( s \sqrt{4 \omega_1^2 - 2\rho_0^2}  \right) }{\sqrt{4 \omega_1^2 - 2\rho_0^2}} + i \omega_1 \frac{ \cos \left( s \sqrt{4 \omega_1^2 - 2\rho_0^2}   \right) -1}{2 \omega_1^2 - \rho_0^2}.$$ 

Clearly, $\zeta(s)= e^{i \omega_1 s} \eta(s)$ has infinitely many conjugate points $\frac{2 \pi n}{\sqrt{4 \omega_1^2 - 2\rho_0^2}}$, $n \in \N$.

Given any interval $I$, the quadratic functional $\tilde{Q}_{1,I}: H_0^1(I, \mathbb{C}) \to \R$, 

$$\tilde{Q}_{1, I}(\sigma_k)= \int_{I} |\sigma_k'|^2 - (1+c^2/4 - |\phi|^2) |\sigma_k|^2 +2 (\lan \sigma_k, \phi \ran)^2 <0$$
is in the conditions of Section 29.2 of \cite{gelfand}. We can apply \cite[Theorem 3' in page 122]{gelfand} to deduce that $\tilde{Q}_{1, I}$ takes negative values as soon as the length of the interval $I$ is greater than $\frac{2 \pi }{\sqrt{4 \omega_1^2 - 2\rho_0^2}}$. Then we can find infinitely many functions $\sigma_k \in C^{\infty}_0(\R)$ \emph{ with disjoint support} such that 

$$\tilde{Q}_1(\sigma_k)= \int_{-\infty}^{\infty} |\sigma_k'|^2 - (1+c^2/4 - |\phi|^2) |\sigma_k|^2 +2 (\lan \sigma_k, \phi \ran)^2 <0.$$

We now want to pass to the original problem \eqref{eq:GP1D} and estimate its Morse index. In order to do so, define $\tau_k(s)$ by $\sigma_k(s)= e^{ics/2}\tau_k(s)$. Simple computations give:

$$ \sigma_k'(s)= (i \frac c 2 \tau_k(s) + \tau_k'(s) )e^{ics/2},$$

$$|\sigma_k'(s)|^2 = |\tau_k'(s)|^2 + \frac{c^2}{4} |\tau_k(s)|^2 + c \lan i \tau_k(s), \tau_k'(s) \ran = |\tau_k'(s)|^2 + \frac{c^2}{4} |\tau_k(s)|^2  - c \lan \tau_k(s), i \tau_k'(s) \ran.$$

Moreover,

$$ \lan \sigma_k(s), \phi(s) \ran = \lan \tau_k(s), \psi(s) \ran.$$

As a consequence $\tilde{Q}_1(\sigma_k)= Q_1(\tau_k)<0$, where

$$Q_1(\tau_k)=\int_{-\infty}^{\infty} |\tau_k'|^2  - c \lan \tau_k, i \tau_k' \ran - (1 - |\psi|^2) |\tau_k|^2 +2 (\lan \tau_k, \psi \ran)^2.$$

Observe that this is the quadratic form associated to \eqref{eq:GP1D}. 

Take now a $C_0^\infty$ function $\chi_k: \R^{d-1} \to \R^+$, and let us estimate $Q$ on the function $\iota_k(x)=\chi_k(\tilde{x}) \tau_k(x_1)$, where $Q$ is defined in \eqref{defQ}:

$$Q(\iota_k)= Q_1(\tau_k) \int_{\R^{d-1}} \chi_k(\tilde{x})^2 \, d\tilde{x} + \left( \int_{-\infty}^{+ \infty} |\tau_k(x_1)|^2 \, dx_1\right)
\left( \int_{\R^{d-1}} |\nabla \chi_k(\tilde{x})|^2 \, d\tilde{x} \right).$$

It suffices to take now $\chi_k$ such that $\int_{\R^{d-1}} \chi_k^2=1$ and $\int_{\R^{d-1}} |\nabla \chi_k|^2 $ is sufficiently small, to conclude that $Q(\iota_k )<0$. 

\medskip

Observe also that $supp \ \iota_k  \cap supp \ \iota_{k'} = \emptyset$ if $k \neq k'$, since an analogue property holds for $\sigma_k$ and $\tau_k$. Hence, $Q$ is negative definite on the vector space generated by the linearly independent functions $\{\iota_1, \ \dots, \iota_k\}$ for any $k \in \N$, concluding the proof.

\end{proof}

\end{document}